\documentclass[reqno,12pt]{amsart}
\usepackage{amsmath,amssymb,amsthm}
\usepackage{float}
\usepackage{geometry}
\usepackage{tabularx}
\usepackage[dvipsnames]{xcolor}
\usepackage[
         breaklinks=true,
		colorlinks=true,
		linkcolor=blue,
		citecolor=ForestGreen,
		urlcolor=cyan
		]{hyperref}
\usepackage{mathtools}\usepackage{soul}

\usepackage{subfigure}\usepackage{tikz}

\newtheorem{lemma}{Lemma}[section]
\newtheorem{AL}{Lemma}[section]
\newtheorem{proposition}{Proposition}[section]

\newtheorem{theorem}{Theorem}[section]

\theoremstyle{definition}

\theoremstyle{remark}

\theoremstyle{remark}
\newtheorem{remark}{Remark}[section]
\numberwithin{equation}{section}

\newcommand{\R}{{\mathbb R}}

\newcommand{\la}{{\mathcal L}_a}
\definecolor{blu}{rgb}{0,0,1}

\newcommand{\rd}{{\mathbb R}^d}

\allowdisplaybreaks
\newtheorem{itheorem}{{\bf Theorem}}[section]

\title[The heat equation with inverse square potential] {On  inhomogeneous heat equation with\\ inverse square potential}

\author[D.  G.  Bhimani]{Divyang G. Bhimani}
\address{Department of Mathematics, Indian Institute of Science Education and Research, Dr. Homi Bhabha Road, Pune 411008, India}
\email{divyang.bhimani@iiserpune.ac.in}

\author[S. Haque] {Saikatul Haque}
\address{Harish-Chandra Research Institute, Chhatnag Road, Jhunsi, 
Prayagraj (Allahabad) 211019, India}
\email{saikatulhaque@hri.res.in}

\begin{document}

\begin{abstract} We study inhomogeneous 
heat equation  with inverse square potential, namely, 
\[\partial_tu +  \mathcal{L}_a u= \pm |\cdot|^{-b} |u|^{\alpha}u,\]
where $\mathcal{L}_a=-\Delta + a |x|^{-2}.$
We establish some fixed-time decay estimate for $e^{-t\mathcal{L}_a}$ associated with inhomogeneous nonlinearity $|\cdot|^{-b}$ in Lebesgue spaces.  We then  develop local theory in  $L^q-$ scaling  critical and super-critical regime  and small data global well-posedness in critical Lebegue spaces.  We further study asymptotic behaviour of global solutions by using self-similar solutions,  provided the initial data satisfies certain bounds. 
 Our method of proof is inspired from  the work of Slimene-Tayachi-Weissler (2017) where they considered the classical case,  i.e.  $a=0$.  
\end{abstract}
\maketitle
\tableofcontents
\section{Introduction}
We study heat equation associated  to inverse square potential:
\begin{equation}\label{0}
\begin{cases}
u_t(t,x) +\mathcal{L}_au(t,x)=\frac{\mu}{|x|^{b}}|u(t,x)|^{\alpha}u(t,x) \\
u(0,x)=\varphi(x)
\end{cases} (t,x) \in \R^+ \times \R^d
\end{equation}
where $d\geq2$, $\mu \in \{ \pm 1 \}$ and $b, \alpha>0.$  \eqref{0} is also known as  Hardy parabolic equation.
The  Schr\"odinger  operator
$$\mathcal{L}_{a}=-\Delta + \frac{a}{|x|^2},$$
with $a\geq- \left(\frac{d-2}{2} \right)^2$,  is initially defined
with domain $C_{c}^{\infty}(\R^d \setminus \{ 0\}).$ See \cite[Section 1.1]{Killipatal}. It is then %can be  
extended   as an unbounded operator in $L^p(\rd)$ that 
generates  a positive semigroup $\{e^{-t\la}\}_{t\geq0}$ in $L^p(\rd)$ for  $s_1<d/p<s_2+2.$ Here, 
\[
s_1:=s_1(a)=\frac{d-2}{2}-\sqrt{\frac{(d-2)^2}{4}+a} \quad \text{and} \quad   s_2:=s_2(a)=\frac{d-2}{2}+\sqrt{\frac{(d-2)^2}{4}+a}
\]
 are the roots of $s^2-(d-2)s-a=0$, see  \cite[Theorems 1.1, 1.3]{metafune2016scale}.  %See also see \cite[Section 3]{pilarczyk2013self}.  
 Moreover, the semigroup $\{e^{-t\la}\}_{t\geq0}$  has the following smoothing effect (see also Remark \ref{rfu} below).

\begin{itheorem}[Decay estimate, Theorem 5.1 in \cite{ioku2016estimates}]\label{ste}
Assume $d\geq2$, $a\geq- \left(\frac{d-2}{2} \right)^2$  and % \textcolor{red}{what is $d$?}
\begin{equation}\label{r}
\tilde{s}_1=\max(s_1,0) \quad \text{and} \quad \tilde{s}_2=\min(s_2,d-2).
\end{equation}
Then,  for all $\tilde{s}_1<\frac{d}{q}\leq\frac{d}{p}<\tilde{s}_2+2$ and $t>0,$ we have 
\begin{equation}\label{d}
\|e^{-t\la}f\|_{L^q}\leq ct^{-\frac{d}{2}(\frac{1}{p}-\frac{1}{q})}\|f\|_{L^p}.
\end{equation}
\end{itheorem}
This smoothing effect  for $e^{-t\la}$ will play  a vital role in our analysis of  short and long time behaviour of (mild) solutions  of \eqref{0}.  We note that the inverse-square potential breaks space-translation symmetry for \eqref{0}. However,   it retains the scaling symmetry.   Specifically, 
if $u (t,x)$ solves \eqref{0}, then 
\begin{equation}\label{scl}
 u_{\lambda} (t,x) = \lambda^{\frac{2-b}{\alpha}} u(\lambda^{2} t, \lambda x) 
\end{equation}
 also solves \eqref{0} with data $u_\lambda(0).$ The Lebesgue space $L^q$ is invariant   under the above scaling   only when   $q=q_c:= \frac{d \alpha}{2-b}.$ We shall see that this constant $q_c$ will play important role in studying \eqref{0}. The problem \eqref{0} %is called %say that \eqref{0}
  is $L^q-$
\begin{equation*}
\begin{cases} \text{sub-critical} & \text{if} \   1\leq q <q_c\\
\text{critical}  & \text{if} \ q=q_c\\
\text{super-critical}  & \text{if} \ q> q_c.
\end{cases}
\end{equation*}

We point out that with $a=0$, the heat equation \eqref{0} is extentively studied, see  \cite{chikami2021well} and the references therein. Corresponding results in the context of Schr\"odinger equation, are investigated in  \cite{guzman2021scattering,bhimani2023sharp} and the references therein.
With $a\neq0$, the operator $\mathcal{L}_a$ is a mathematically intriguing borderline situation that can be found in a number of physical contexts,  including geometry,  combustion theory to the Dirac equation with Coulomb potential,  and to the study of perturbations of classic space-time metrics such as Schwarzschild and Reissner–Nordstr\"om.  See \cite{Killipatal, Zhang, Burq, Kalf, Luis} for detailed discussions.
In fact,   substantial  progress has been made to understand  well-poseness theory for  nonlinear Schr\"odinger and wave  equation associated to $\mathcal{L}_a$,  see for e.g. \cite{Jason, Haque,  Killip, Zhang} and references therein.   The mathematical interest in these equations with $a|x|^{-2}$ however comes mainly from the fact that the potential term is homogeneous of degree -2 and therefore scales exactly the same as the Laplacian.  On the other hand,  it appears that  we know very little   about  the  well-posedness results  for  heat equation  associated to $\mathcal{L}_a,$  even when $b=0$ in \eqref{0}. In this note,  we aim to  initiate  a systematic study  of well-posedness theory for  \eqref{0}.

In this article, by a solution to \eqref{0} we mean mild solution of \eqref{0},  that is,  a solution of the integral equation
\begin{equation*}%\label{1}
u(t)=e^{-t\la}\varphi+\mu\int_0^t e^{-(t-s)\la}(|\cdot|^{-b}|u(s)|^\alpha u(s))ds.
\end{equation*} 

We now state our first theorem.
\begin{theorem}[Local theory]\label{local} Let $d\geq2$, $a\geq-\frac{(d-2)^2}{4}$, $0\leq b<\min(2,d),$  $\tilde{s}_1,\tilde{s}_2$ be as in \eqref{r} and $0<\alpha<\frac{2-b}{\tilde{s}_1}$   and $\varphi\in L^q(\rd).$ 
\begin{enumerate}
\item \label{i}
Assume that  $$ \max\left\{\frac{d(\alpha+1)}{\tilde{s}_2+2-b},q_c\right\}<q<\frac{d}{\tilde{s}_1}.$$ 
Then there exists a maximal time $T_{\max}>0$ and a unique solution $u$ of \eqref{0} such that $$u \in C([0,T_{\max} ); L^q (\rd )).$$ Moreover,  blowup alternative holds,  i.e.  if $T_{\max}<\infty,$
 then  $\lim_{t\to T_{\max}} \|u(t)\|_q = \infty$.

\item \label{ii} Assume that  
$$q_c   \leq q<  \frac{d}{\tilde{s}_1},\qquad q>\frac{d}{\tilde{s}_2+2}$$ 
(cf. Figure \ref{f2}).   Then\begin{itemize}
\item[(i)] there exist $T > 0$ and a  solution $u$ of \eqref{0} such that $u \in C ([0, T ], L^q (\rd )).$ 
\item[(ii)] uniqueness in part (i) is guaranteed only among functions in
\begin{itemize}
\item $\{u \in C([0,T] ; L^q(\rd):\sup_{t\in[0,T ]} t^{\frac{d}{2}(\frac{1}{q}-\frac{1}{r})}\|u(t)\|_r<\infty\}$ where $r$ satisfies \eqref{2} and $q>q_c.$
\item $\{u \in C([0,T] ; L^q(\rd):\sup_{t\in[0,T ]} t^{\frac{d}{2}(\frac{1}{q_c}-\frac{1}{r})}\|u(t)\|_r\leq M \}$ for some $M>0$ and $r$ satisfies \eqref{6} and $q=q_c$.
\end{itemize} 
\end{itemize}   Moreover,  for the super-critical exponents, continuous dependency on data and the blowup alternative holds.

\item\label{iii}  In all the above cases, except where $q = q_c$, the maximal existence time of the solution, denoted by $T_{\max}$, depends on $\|\varphi\|_q$, not $\varphi$ itself.
\end{enumerate}
\end{theorem}

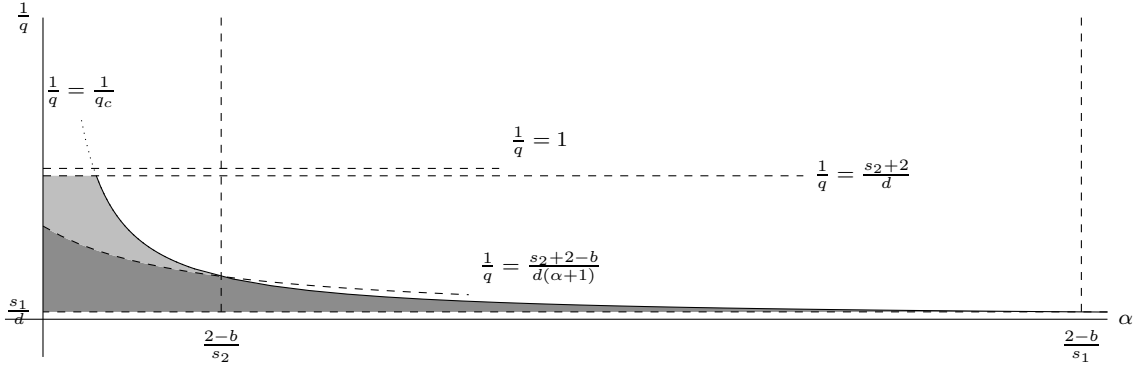
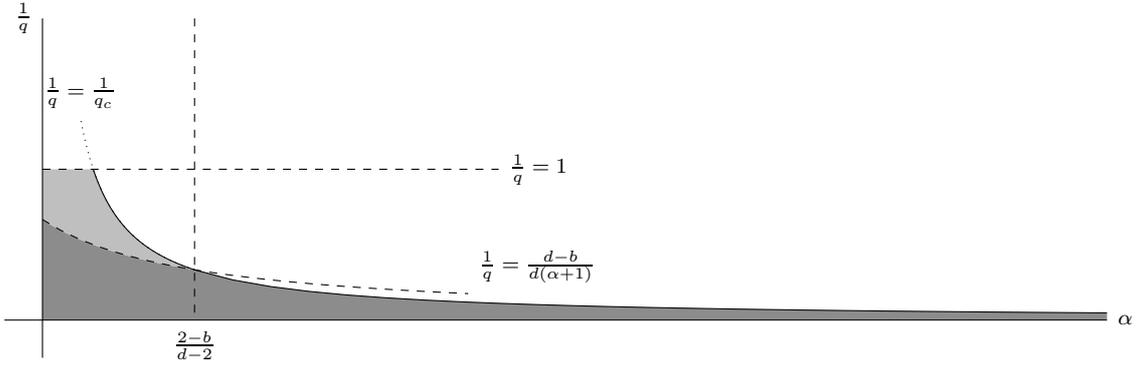
\begin{figure}
\subfigure[The case  $d=3$, $a=-\frac{1}{8}$, $b=1$] {
\begin{tikzpicture}[scale=2]

\fill [gray!90!white](0,1/6-2^.5/12)--(0,5/3-2*2^.5/3)--plot[domain=0:4-2*2^.5]  ({\x},{(6+2^.5)/(12*\x+12)})--plot[domain=4-2*2^.5:4+2*2^.5]  ({\x},{(1/3)*(1/\x)})--(4-2*2^.5,1/6-2^.5/12);
\fill [gray!50!white](0,1/6-2^.5/12)--(0,10/12+2^.5/12)--(1/3,10/12+2^.5/12)--plot[domain=20/49-2*2^.5/49:4-2*2^.5]  ({\x},{(1/3)*(1/\x)})--plot[domain=4-2*2^.5:0]  ({\x},{(6+2^.5)/(12*\x+12)});
\draw[][] (-.25,0)--(7,0) node[anchor=west] {\tiny{$\alpha$}};
\draw[][-] (0,-.25)--(0,2) node[anchor=east] {\tiny{$\frac{1}{q}$}};
\draw[dashed](0,1)--(3,1) node[anchor=south west] {\tiny{$\frac{1}{q}=1$}};
\draw[dashed](4+2*2^.5,2)--(4+2*2^.5,0) node[anchor=north] {\tiny{$\frac{2-b}{s_1}$}};
\draw[dashed](4-2*2^.5,2)--(4-2*2^.5,0) node[anchor=north] {\tiny{$\frac{2-b}{s_2}$}};
\draw[dotted][-] [domain=1:.25] plot ({\x}, {(1/3)*(1/\x)})node[anchor=south] {\tiny{{$\frac{1}{q}=\frac{1}{q_c}$}}};
\draw[][-] [domain=1:20/49-2*2^.5/49] plot ({\x}, {(1/3)*(1/\x)});%node[anchor=south] {\tiny{{$\frac{1}{q}=\frac{1}{q_c}$}}};
\draw[][-] [domain=7:1] plot ({\x}, {(1/3)*(1/\x)});
\draw[dashed][] [domain=0:2.8] plot ({\x}, {(6+2^.5)/(12*\x+12)})node[anchor=south west] {\tiny{{$\frac{1}{q}=\frac{s_2+2-b}{d(\alpha+1)}$}}};
\draw[dashed](7,1/6-2^.5/12)--(0,1/6-2^.5/12) node[anchor=east] {\tiny{$\frac{s_1}{d}$}};
\draw[dashed](0,10/12+2^.5/12)--(5,10/12+2^.5/12) node[anchor=west] {\tiny{$\frac{1}{q}=\frac{s_2+2}{d}$}};
\end{tikzpicture}
}
\subfigure[The case $d=3$, $a\geq0$, $b=1$]  {
\begin{tikzpicture}[scale=2]
\fill [gray!90!white](0,0)--(0,2/3)--plot[domain=0:1]  ({\x},{2/(3*\x+3)})--plot[domain=1:7]({\x},{1/(3*\x)})--(7,0);
\fill [gray!50!white](0,2/3)--(0,1)--(1/3,1)--plot[domain=1/3:1]  ({\x},{1/(3*\x)})--plot[domain=1:0]  ({\x},{2/(3*\x+3)});
\draw[][] (-.25,0)--(7,0) node[anchor=west] {\tiny{$\alpha$}};
\draw[][-] (0,-.25)--(0,2) node[anchor=east] {\tiny{$\frac{1}{q}$}};
\draw[dashed](0,1)--(3,1) node[anchor=west] {\tiny{$\frac{1}{q}=1$}};
\draw[dashed](1,2)--(1,0) node[anchor=north] {\tiny{$\frac{2-b}{d-2}$}};
\draw[dotted][-] [domain=1:.25] plot ({\x}, {(1/3)*(1/\x)})node[anchor=south] {\tiny{{$\frac{1}{q}=\frac{1}{q_c}$}}};
%\draw[dashed][-] [domain=7:1] plot ({\x}, {(1/3)*(1/\x)});
\draw[dashed][] [domain=0:2.8] plot ({\x}, {(2/3)*1/(\x+1)})node[anchor=south west] {\tiny{{$\frac{1}{q}=\frac{d-b}{d(\alpha+1)}$}}};
\draw[][-] [domain=7:1] plot ({\x}, {(1/3)*(1/\x)});
\draw[][-] [domain=1:1/3] plot ({\x}, {(1/3)*(1/\x)});%node[anchor=south] {\tiny{{$\frac{1}{q}=\frac{1}{q_c}$}}};
\end{tikzpicture}
}
\caption{\tiny{Local well-posedness in mere $L^q(\rd)$ occurs in the darkly shaded (open) region by part \eqref{i} of Theorem \ref{local}. Local existence in $L^q(\rd)$ is guaranteed by part \eqref{ii} of Theorem \ref{local}  in the total shaded (dark \& light) region along with the open segment on boundary which is part of the curve $\frac{1}{q}=\frac{1}{q_c}$.}}  \label{f2}
\end{figure}

As far as we know,  Theorem \ref{local} is new for $a\neq 0.$
In \cite[Theorem 1.1]{slimene2017well},  Slimene,  Tayachi and Weissler  proved 
Theorem \ref{local} when $a=0.$  Our method of proof is based on standard contraction argument and inspired from the work in \cite{slimene2017well}.  The main key ingredient  to prove Theorem \ref{local} is Proposition \ref{est0}.  In order to prove Proposition \ref{est0},   Theorem \ref{ste} and scaling properties of $e^{-t\la}$ %invariance of \eqref{0} 
play a vital role.

\begin{remark}\label{r1} For $a=0$,   \eqref{d} is valid for the end points,  see Remark \ref{rfu} below.   Thus, in this case,  whenever there is a strict inequality $<$ involving $s_1$ in Theorem \ref{local} it can be relaxed to non strict one $\leq$. 
\end{remark}

\begin{remark} The case $a=b=0$ corresponds to standard nonlinear heat equation.  This has been extensively studied in the literature  since the pioneering work  done in  early 80s, see \cite{weissler1980local, brezis1996nonlinear}. 
In this situation,  for the sub-critical case,  Weissler \cite{weissler1980local} proved local well-poseness for \eqref{0}   in $L^q(\R^d)$  for $q> q_c\geq 1.$  For sub-critical case,   there is no general theory of existence,  see \cite{weissler1980local, brezis1996nonlinear}. In fact, Haraux-Weissler \cite{haraux1982non} proved  that if  $1<q_c <  \alpha +2$, then  there is a global solution (with zero initial data) in $L^q(\R^d)$ for  $1\leq q < q_c.$ But no such  solution exists when $\alpha +2 < q_c.$  For critical case,  i.e.  $q=q_c$  the solution exists globally in time for small initial data.  
\end{remark}
\begin{remark} For the sub-critical  exponents,  i.e.  for $q<q_c,$ classical inhomogeneous heat equation,  i.e.  \eqref{0} with $a=0,$ is known to be  ill-posed in $L^q(\R^d)$.   We believe that similar result hold for \eqref{0}.  However,  shall not pursue this issue in the present paper.
\end{remark}

After achieving local well-posedness with a contraction mapping,  one then finds the following lower estimate for blow-up in the super critical case $q>q_c$ as in the classical case $a=b=0$.
\begin{theorem} [Lower blow-up rate]  Assume that $q> \max (1, q_c)$ and $T_{\max}< \infty,$ where $T_{\max}$ is the existence time of  the resulting maximal solution of \eqref{0}.  Then
under the hypotheses of Theorem \ref{local},  we have 
       $$\|u(t)\|_q \gtrsim (T_{max}-t)^{\frac{ d}{2q}-\frac{2-b}{2\alpha}} ,\quad \forall t\in[0,T_{max}).$$
\end{theorem}

On the other hand in the critical case $q=q_c$, we achieve `small' data global well-posedness where the smallness is in the sense described in the theorem below:
\begin{theorem}[Small data global existence]\label{global} Let $d \geq 2,$ $0\leq b<\min(2,d)$ and $\frac{d}{\tilde{s}_2+2}<q_c<\frac{d}{\tilde{s}_1}$ i.e. $$\frac{2-b}{\tilde{s}_2+2}<\alpha<\frac{2-b}{\tilde{s}_1}.$$
Then we have the following.\begin{enumerate}
\item\label{global-1}  If $\varphi \in L^{q_c} (\rd )$ and $\|\varphi\|_{q_c}$ is sufficiently small, then $T_{max} (\varphi)= \infty.$
\item\label{global-2} If $\varphi\in \mathcal{S}'(\rd)$ such that $|\varphi|\leq c( 1+|\cdot|^2)^{-\sigma/2}$, %for all $x\in\rd$ with 
$c $ sufficiently small and $\sigma>\frac{2-b}{\alpha}$, then $T_{max}(\varphi) = \infty$.
\item\label{global-3}
Let $\varphi\in \mathcal{S}' (\rd )$ be such that $|\varphi| \leq c|\cdot|^{- \frac{2-b}{\alpha}}$ , for $c$  sufficiently small. Then there exists a global 
time solution of (1.5), $u\in C([0,\infty);L^q(\rd))$ for all $q\in(q_c,\frac{d}{\tilde{s}_1})$. Moreover $u(t)\to\varphi$ in $\mathcal{S}'(\rd)$ as $t\to0$.
\end{enumerate}
\end{theorem}

Theorem \ref{global} is proved using  more general  Theorem \ref{global2} below.    This is inspired from  \cite[Theorem 4.1]{slimene2017well} (that deal with $a=0$ case),  in  which  they uses idea from the earlier work \cite[Theorem 6.1]{cazenave1998asym} (that deals with $a=b=0$ case).

A  solution of  \eqref{0} is \textit{self-similar} if $u_\lambda = u$ for all $\lambda > 0$ where $u_\lambda$ is defined  in \eqref{scl}.  In \cite{Hirose,  Wang},  authors have established the existence of radially symmetric self-similar solutions and later in \cite[Theorem 1.4]{slimene2017well}  authors have proved the  self-similar solutions that are not necessarily symmetric  for classical  inhomogeneous heat equation.  In the next theorem,  we establish  similar result in the presence of  inverse square potential.
\begin{theorem}[Self-similar solutions]\label{selfsimilar}
Let $d\geq2,$ $0 < b < \min(2, d )$ and $\frac{d}{\tilde{s}_2+2}<q_c<\frac{d}{\tilde{s}_1}$ i.e. $$\frac{2-b}{\tilde{s}_2+2}<\alpha<\frac{2-b}{\tilde{s}_1}.$$
Let $\varphi(x) = \omega(x)|x|^{-\frac{2-b}{\alpha}}$, where $\omega \in L^\infty(\rd)$
 is homogeneous of degree $0$ and $\|\omega\|_\infty$  is sufficiently small. Then there exists a global mild self-similar solution $u_S$ of \eqref{0} and $u_S(t)\to\varphi$ in $\mathcal{S}'(\rd)$ as $t\to 0$.
\end{theorem}
Using this self-similar solution,  the next theorem gives the information about  the asymptotic behaviour of  global solutions achieved by Theorem \ref{global}, provided the data satisfies certain bounds.
\begin{theorem}[Asymptotic behaviour]\label{asym}
Let $d\geq2,$  $0 < b < \min(2, d )$ and $\frac{d}{\tilde{s}_2+2}<q_c<\frac{d}{\tilde{s}_1}$ i.e. $\frac{2-b}{\tilde{s}_2+2}<\alpha<\frac{2-b}{\tilde{s}_1}$ and
\[
\frac{2-b}{\alpha}\leq\sigma<\tilde{s}_2+2.
\]
Let $\varphi\in \mathcal{S}'(\rd)$ be such that
\[
|\varphi(x)|\leq c( 1+|x|^2)^{-\sigma/2}, \quad \forall x\in\rd
\]
for $c > 0$ sufficiently small,  and 
\[|\varphi(x)|= \omega(x)|x|^{-\sigma}, \quad \forall |x|\geq A \]
for some constant $A>0$ and  some $\omega\in L^\infty(\rd)$ homogeneous of degree $0$ with $\|\omega\|_\infty$ sufficiently small.

Let $u,u_{\mathcal{S}}$ be the unique solutions to \eqref{0} with data $\varphi$, $\omega(x)|x|^{-\frac{2-b}{\alpha}}$ given by Theorem \ref{global}, Theorem \ref{selfsimilar} respectively. Then
\begin{enumerate}
\item\label{asym1} if $\sigma=\frac{2-b}{\alpha}$ and $q\in[r,\frac{d}{\tilde{s}_1})$  there exists $\delta > 0$ such that  \begin{equation}\label{p10}
\|u(t)-u_{\mathcal{S}}(t)\|_q \leq Ct^{-(\frac{\sigma}{2}-\frac{d}{2q})-\delta}, \text{ for all }t>0
\end{equation}
where $C$ is a positive constant.
In particular, if $\omega \neq 0$, there exist $c_1, c_2>0$  such that for $t$ large
\[
c_1 t^{-(\frac{2-b}{2\alpha}-\frac{d}{2q})} \leq \|u(t)\|_q \leq  c_2 t^{-(\frac{\sigma}{2}-\frac{d}{2q})}.
\]

\item\label{asym2}  if $\frac{2-b}{\alpha}<\sigma<(2-b)(\frac{\tilde{s}_2+2-b}{\tilde{s}_1\alpha}-1)$ and $q\in[r_1,\frac{d}{\tilde{s}_1})$ where $r_1$ as in Lemma \ref{11}, there exists $\delta > 0$ such that 
$$\|u(t)-e^{-t\la}(\omega(x)|x|^{-\sigma})\|_q \leq Ct^{-(\frac{\sigma}{2}-\frac{d}{2q})-\delta}
%, \text{ for all }t\geq t_q
$$
for all $t$ large enough. 
In particular, if $\omega \neq 0$, there exist $c_1, c_2>0$  such that for $t$ large
\[
c_1 t^{-(\frac{\sigma}{2}-\frac{d}{2q})} \leq \|u(t)\|_q \leq  c_2 t^{-(\frac{\sigma}{2}-\frac{d}{2q})}.
\]
\end{enumerate}
\end{theorem}
%\textcolor{blue}{why can't we recover $q=\infty$ case,  as the classical case, should be recover, i.e without potential ?}

\begin{remark}
As remarked earlier, see Remark \ref{r1}, when $a=0$, one can actually gets the above results even for $q=\frac{d}{\tilde{s}_1}=\infty$, recovering the result in \cite{{slimene2017well}}.
\end{remark}

\begin{remark}
In the case \eqref{asym1}, since $\frac{\sigma}{2}-\frac{d}{2q}=\frac{2-b}{2\alpha}-\frac{d}{2q}>0$, it follows from \eqref{p10}, that for $t$ large, the solution is close to a nonlinear self-similar solution. Thus in  this case, the solution has  nonlinear behaviour near $t=\infty$. By similar reasoning, in the case \eqref{asym2}, the solution  has linear behaviour near $t=\infty$, which can be related to scattering phenomenon. In both the case the final conclusion says $\|u(t)\|_q\to0$ as $t\to\infty$.
\end{remark}

\begin{remark}
The method of proof of Theorem \ref{asym} \eqref{asym2} differs at certain stage from the one in \cite{slimene2017well} (that deals with $a=0$ case) due the fact that in the case $a\neq0$, one do not have the decay estimate \eqref{d} for $q=\infty$, see Remark \ref{r2}. 
\end{remark}
\begin{remark}\label{rfu} \
\begin{enumerate}
\item  Notice that 
\begin{equation*}
\tilde{s}_1=\tilde{s}_1(a)=\begin{cases}
s_1(a)&\text{ for }a\in[-\frac{(d-2)^2}{4},0)\\
0 &\text{ for }a\in[0,\infty)
\end{cases},
\qquad
\tilde{s}_2=\tilde{s}_2(a)=\begin{cases}
s_2(a)&\text{ for }a\in[-\frac{(d-2)^2}{4},0)\\
d-2 &\text{ for }a\in[0,\infty)
\end{cases}.
\end{equation*}
\item\label{rfu2}
 For $a=0,$ 
Theorem \ref{ste} says \eqref{d} valid for $1<p\leq q<\infty$.  However,  since $e^{-t\mathcal{L}_0}f=e^{t\Delta}f=k_t*f$ where $k_t(x)=(4\pi t)^{-d/2}\exp(-|x|^2/(4t))$,  the inequality \eqref{d}  holds for all $1\leq p\leq q\leq\infty$: In fact %including the end points $p=1$, $q=\infty$:  
by Young's inequality, $\|e^{t\Delta}f\|_q\leq\|k_t\|_r\|f\|_p=t^{-\frac{d}{2r'}}\|k_1\|_r\|f\|_p$  with $\frac{1}{r}=1+\frac{1}{q}-\frac{1}{p}$ i.e. $\frac{1}{r'}=\frac{1}{p}-\frac{1}{q}$  (assumption $p\leq q$ is  to make sure %
 $r\geq1$).% and $\|K_t\|_r=t^{-\frac{d}{2r'}}\|K_1\|_r$  %,See for e.g.  \cite{}.  \textcolor{red}{this should also hold for $q=\infty,  p=1$? See the paper of Miao}
\item 
For $-\frac{(d-2)^2}{4}\leq a<0$, 
from Theorem \ref{ste}, \eqref{d} is valid for $1<\frac{d}{s_2+2}<p\leq q<\frac{d}{s_1}<\infty$. Therefore in this case, we get results valid with more restrictions (involving $s_1,s_2$) compared to the case $a=0$.  For example, for the local theory, compare the regions   for the case (a) and the case (b)  in Figure \ref{f2}.

\item On the other hand, for $a>0$, %one has $s_1<0$ and $s_2+2>d$. Therefore 
from Theorem \ref{ste}, \eqref{d} is valid for all $1<p\leq q<\infty$. %as in the case $a=0$. 
Thus in this case the results match with the results for case $a=0$ (and do not have restrictions in involving $s_1,s_2$) except the end point restrictions. 
\end{enumerate} 
\end{remark}

The paper is organized as follows.  In Section \ref{keat},  we prove key estimate 
for $e^{-t\la}(|\cdot|^{-b} f)$.   In Section \ref{es},  we prove Theorems \ref{local}, \ref{global} and \ref{selfsimilar}.  In Section \ref{AB},  we prove Theorem \ref{asym}. \\
\noindent
\textit{Notations}.  The notation $A \lesssim B $ means $A \leq cB$ for some universal constant $c > 0.$ The Schwartz space is denoted by  $\mathcal{S}(\mathbb R^{d})$,  and the space of tempered distributions is  denoted by $\mathcal{S'}(\mathbb R^{d}).$ 
\section{Key  Ingradient}\label{keat}
In order to incorporate the inhomogeneous nonlinearity,  we first establish  some fixed-time estimate for $e^{-t\mathcal{L}_a}|\cdot|^{-b}$ in the next proposition.
\begin{proposition}\label{est0}
Let $d\geq2$, $a\geq-\frac{(d-2)^2}{4}$, $0\leq b<d$ and
 \begin{equation}\label{p1}
\tilde{s}_1<\frac{d}{q}\leq b+\frac{d}{p}<\tilde{s}_2+2%\frac{s_1}{d}< \frac{1}{q}<\frac{b}{d}+\frac{1}{p}<\frac{s_2+2}{d}
\end{equation} 
then for $t>0$
 \[
 \|e^{-t\la}(|\cdot|^{-b}f)\|_{L^q}\leq ct^{-\frac{d
 }{2}(\frac{1}{p}-\frac{1}{q})-\frac{b}{2}}\|f\|_{L^p}.
 \]
\end{proposition}
\begin{proof}
Assume first that $t=1$. Put $m=\frac{d}{b}.$ Since $\frac{1}{q}<\frac{1}{m}+\frac{1}{p}<\frac{\tilde{s}_2+2}{d},$  we can choose $\epsilon,\delta>0$ so that 
\begin{equation}\label{p2}
\frac{1}{q}<\frac{1}{m+\delta}+\frac{1}{p}<\frac{1}{m-\epsilon}+\frac{1}{p}<\frac{\tilde{s}_2+2}{d}.
\end{equation}
 Split $|\cdot|^{-b}$ as follows
\[
|\cdot|^{-b}=k_1+k_2 \quad\text{with }k_1\in L^{m-\epsilon}(\rd), \quad k_2\in L^{m+\delta}(\rd),
\]
%\textcolor{red}{Why above splitting possible? may be }
for example one may take
$$k_1= \chi_{ \{|x|\leq 1 \}} |\cdot|^{-b},  \quad k_2= \chi_{ \{|x|> 1 \}} |\cdot|^{-b}.    $$ 
%{it might be a good idea to be precise here?}
Let $$\frac{1}{r_1}=\frac{1}{m-\epsilon}+\frac{1}{p}, \quad \frac{1}{r_2}=\frac{1}{m+\delta}+\frac{1}{p}$$ and note 
 that $\tilde{s}_1<\frac{d}{q}<\frac{d}{r_i}<\tilde{s}_2+2$. %$\frac{2d}{d+2}<r_i\leq q<\frac{2d}{d-2}.$
By Theorem \ref{ste} and H\"older's inequality,  we have 
\begin{eqnarray*}
 \|e^{-\la}(|\cdot|^{-b}f)\|_q&\leq&\|e^{-\la}(k_1f)\|_q+\|e^{-\la}(k_2f)\|_q\\
 &\lesssim&\|k_1f\|_{r_1}+\|k_2f\|_{r_2}\\
 &\leq&(\|k_1\|_{m-\epsilon}+\|k_2\|_{m+\delta})\|f\|_{p}\lesssim\|f\|_p.
\end{eqnarray*}
Thus the case $t=1$ is proved.

%\textcolor{red}{it would be a good idea to use $\gamma$ for the scaling purposes because we used $\lambda$ for some other purposed in the beginning}
 
 For $\varphi\in\mathcal{S}(\rd)$, set $D_\lambda\varphi=\varphi(\lambda\ \cdot)$. Then we claim  that  $$e^{-t\la}(D_\lambda\varphi) =D_\lambda(e^{-\lambda^2t\la}\varphi).$$
%\textcolor{red}{probably in the above it should be  $D_{\lambda}(e^{-\lambda^2t\la}\varphi )$ } 
 In fact if $v(t)=e^{-t\la}(D_\lambda\varphi)$, then $v$ solves
\begin{equation}\label{si}
\begin{cases} v_t=\la v\\
v(0,x)=D_\lambda\varphi(x)=\varphi(\lambda x).
\end{cases}
\end{equation} 
Put  $w(t)=u(\lambda^2t)$ where $u(t)=e^{-t\la}\varphi$.  Then
 we have 
\begin{eqnarray*}
(D_\lambda w(t))_t & = & D_\lambda w_t(t)\\
& = & \lambda^2D_\lambda u_t(\lambda^2t)\\
&= & \lambda^2D_\lambda \la u(\lambda^2t) \quad  (\because  u_t=\la u)\\
&= & \lambda^2D_\lambda \la w(t)\\
& = & \lambda^2(\la w(t))(\lambda\cdot)\\
& = & \la(w(t,\lambda \cdot))   \\
& = & \la(D_\lambda w(t)),
\end{eqnarray*}
and 
$$(D_\lambda w)(0)=(D_\lambda u)(0)=D_\lambda \varphi.$$
 Thus,  $D_\lambda w$ also satisfies \eqref{si}.
 Therefore,  by uniqueness of solution of Banach valued ODE (semigroup theory), one has $$v(t)=D_\lambda w(t).$$ This establishes the claim.

 Using the above claim we have $e^{\lambda^2t\la }\varphi=D_\lambda^{-1}e^{-t\la}(D_\lambda\varphi)=D_{\lambda^{-1}}e^{-t\la}(D_\lambda\varphi)$, putting $\lambda=1/\sqrt{t}$ one has $e^{\la}\varphi=D_{\sqrt{t}}e^{-t\la}(D_{1/\sqrt{t}}\varphi)$. Then from the case $t=1$ it follows that
 \[
 \|D_{\sqrt{t}}e^{-t\la}D_{1/\sqrt{t}}(|\cdot|^{-b}\varphi)\|_q\lesssim\|\varphi\|_p.
 \]
Since $D_\lambda|\cdot|^{-b}=\lambda^{-b}|\cdot|^{-b}$, we get,
\[
t^{\frac{b}{2}} \|D_{\sqrt{t}}e^{-t\la}(|\cdot|^{-b}D_{1/\sqrt{t}}\varphi)\|_q\lesssim\|\varphi\|_p.
 \]
Replacing $\varphi$ by $D_{\sqrt{t}}\varphi$ we have 
 \[
t^{\frac{b}{2}} \|D_{\sqrt{t}}e^{-t\la}(|\cdot|^{-b}\varphi)\|_q\lesssim\|D_{\sqrt{t}}\varphi\|_p
 \]which implies
 \[
 t^{\frac{b}{2}-\frac{d}{2q}} \|e^{-t\la}(|\cdot|^{-b}\varphi)\|_q\lesssim t^{-\frac{d}{2p}}\|\varphi\|_p
 \]using $\|D_\lambda f\|_p=\lambda^{-d/p}\|f\|_p$. This completes the proof.
\end{proof}

\begin{remark}
In view of Remark \ref{rfu} \eqref{rfu2}, first strict inequality namely $\tilde{s}_1<\frac{d}{q}$  in \eqref{p1} can be relaxed to $\tilde{s}_1\leq\frac{d}{q}$ when $a=0$. But same cannot be done for the last inequality $b+\frac{d}{p}<\tilde{s}_2+2$ as a continuity argument is used to achieve \eqref{p2}.
\end{remark}
\begin{remark} When $a=0$, the above result holds even in dimension $d=1$. Then restriction for $a\neq0$ is due to the fact that the operator $\la$ is not defined on full $\R$.
%\textcolor{red} {You may add remark when $a=0$ the restriction on the dimension is not necessary in Proposition \ref{est0} ?}
\end{remark}
\section{Existence of Solution}\label{es}
\subsection{Local Existence}
\begin{proof}[{\bf Proof of Theorem \ref{local} \eqref{i}}]
Let $K_t(\varphi)=e^{-t\la}(|\cdot|^{-b}|\varphi|^\alpha \varphi)$.  Since 
$\frac{d(\alpha+1)}{\tilde{s}_2+2-b}<q<\frac{d}{\tilde{s}_1},$ we have  \[
\tilde{s}_1<\frac{d}{q}<b+\frac{d}{q/(\alpha+1)}<\tilde{s}_2+2.
\] 
By  Proposition \ref{est0},  we may obtain
\begin{eqnarray*}
\|K_t(\varphi)-K_t(\psi)\|_q&\lesssim&t^{-\frac{d
 }{2}(\frac{\alpha+1}{q}-\frac{1}{q})-\frac{b}{2}}\||\varphi|^\alpha \varphi-|\psi|^\alpha \psi\|_{\frac{q}{\alpha+1}}\\
 &\lesssim &t^{-\frac{d\alpha}{2q}-\frac{b}{2}}\|(|\varphi|^\alpha+|\psi|^\alpha)|\varphi-\psi|\|_{\frac{q}{\alpha+1}}\\
 &\lesssim &t^{-\frac{d\alpha}{2q}-\frac{b}{2}}(\|\varphi\|_q^\alpha+\|\psi\|_q^\alpha)\|\varphi-\psi\|_q\\
&\lesssim &t^{-\frac{d\alpha}{2q}-\frac{b}{2}}M^\alpha\|\varphi-\psi\|_q,
\end{eqnarray*}
provided $\|\varphi\|_q,\|\psi\|_q\leq M$. Note that 
\begin{itemize}
\item $K_t:L^q(\rd)\to L^q(\rd)$ is locally Lipschitz with Lipschitz constant $C_M(t)=t^{-\frac{d\alpha}{2q}-\frac{b}{2}}M^\alpha$ in $\{\varphi\in L^q(\rd):\|\varphi\|\leq M\}$
\item %Lipschitz constant $C_M(t)=t^{-\frac{d\alpha}{2q}-\frac{b}{2}}M^\alpha$  in ball of radius $M$ is in 
%as function of $t$, 
$C_M\in L^1(0,\epsilon)$ as $\frac{d\alpha}{2q}+\frac{b}{2}<1,$ i.e. $q>\frac{d\alpha}{2-b}=q_c$ 
\item $e^{-s\la }K_t=K_{s+t}$ %\textcolor{blue}{may be minus sign}
\item $t\mapsto K_t(0)\equiv0\in L_t^1(0,\epsilon).$
\end{itemize}
By \cite[Theorem 1, p. 279]{weissler1979semilinear}, the result follows. 
In order to ensure room for $q$, in between $\max(\frac{d(\alpha+1)}{\tilde{s}_2+2-b},q_c)$ and $\frac{d}{\tilde{s}_1},$ the hypothesis $0< \alpha < \frac{2-b}{\tilde{s}_1}$ is  imposed,  while $0<b< 2$ is to make $q_c>0$.
\end{proof}
\begin{proof}[{\bf Proof of Theorem \ref{local} \eqref{ii}}] {\bf  Part (i) with the case $q>q_c$.} (the case $q=q_c$ will be treated along with the global existence proof.)

Since $\alpha < \frac{2-b}{\tilde{s}_1}< \frac{2-b}{\tilde{s}_1}+ ( \frac{\tilde{s}_2}{\tilde{s}_1}-1)$ and $\frac{d}{\tilde{s}_2+2}<q<\frac{d}{\tilde{s}_1},$ we have \[
\max\left(\tilde{s}_1,\frac{{d}/{q}-b}{\alpha+1}\right)<\min\left(\frac{\tilde{s}_2+2-b}{\alpha+1},\frac{d}{q}\right).
\] Therefore, one can choose $r$ such that\[
\max\left(\tilde{s}_1,\frac{{d}/{q}-b}{\alpha+1}\right)<\frac{d}{r}<\min\left(\frac{\tilde{s}_2+2-b}{\alpha+1},\frac{d}{q}\right),
\] which implies
\begin{equation}\label{2}
 \tilde{s}_1<\frac{d}{r}<\frac{d}{q}<b+\frac{d(\alpha+1)}{r}<\tilde{s}_2+2.
\end{equation}
Thus Proposition \ref{est0} can be used with exponent pairs $(\frac{r}{\alpha+1},q)$ and $(\frac{r}{\alpha+1},r)$. From $q>q_c$, it follows that
\[
\frac{d(\alpha+1)}{q}-2<\frac{d}{q}-b.
\]
This and \eqref{2} implies that $\frac{d(\alpha+1)}{q}-2<\frac{d}{q}-b<\frac{d(\alpha+1)}{r}$ which gives $d(\frac{1}{q}-\frac{1}{r})(\alpha+1)<2$ i.e. $\beta(\alpha+1)<1$ where $\beta=\frac{d}{2}(\frac{1}{q}-\frac{1}{r})$.
Introduce the space
$$B_M^T=\{u\in C([0,T];L^q(\rd))\cap C((0,T];L^r(\rd)):\max[\sup_{t\in [0,T ]} \|u(t)\|_q, \sup_{t\in[0,T ]} t^\beta\|u(t)\|_r]\leq M\}.$$ This space is endowed with the metric
\[
d(u,v)=:\max[\sup_{t\in [0,T ]} \|u(t)-v(t)\|_q, \sup_{t\in[0,T ]} t^\beta\|u(t)-v(t)\|_r].
\]
Consider the mapping 
\[
\mathcal{J}_\varphi(u)(t)=e^{-t\la}\varphi+\mu\int_0^t e^{-(t-s)\la}(|\cdot|^{-b}|u(s)|^\alpha u(s))ds.
\]
Let $\varphi,\psi\in L^q(\rd)$ and $u,v\in B_M^T.$  By  Proposition \ref{est0}  with exponent pairs $(\frac{r}{\alpha+1},q)$  we get
\begin{eqnarray*}
&&\|\mathcal{J}_\varphi(u)(t)-\mathcal{J}_\psi(v)(t)\|_q\\
%&\leq&\|e^{-t\la}\varphi-e^{-t\la}\psi\|_q+|\mu|\int_0^t \|e^{-(t-s)\la}(|\cdot|^{-b}[|u(s)|^\alpha u(s)-|v(s)|^\alpha v(s)])\|_qds\\
&\lesssim&\|\varphi-\psi\|_q+\int_0^t(t-s)^{-\frac{d}{2}(\frac{\alpha+1}{r}-\frac{1}{q})-\frac{b}{2}}\||u(s)|^\alpha u(s)-|v(s)|^\alpha v(s)\|_{\frac{r}{\alpha+1}}ds\\
&\lesssim&\|\varphi-\psi\|_q+\int_0^t(t-s)^{-\frac{d}{2}(\frac{\alpha+1}{r}-\frac{1}{q})-\frac{b}{2}}(\|u(s)\|_r^\alpha+\|v(s)\|_r^\alpha)\| u(s)- v(s)\|_rds\\
&\lesssim&\|\varphi-\psi\|_q+M^\alpha\int_0^t(t-s)^{-\frac{d}{2}(\frac{\alpha+1}{r}-\frac{1}{q})-\frac{b}{2}}s^{-\beta\alpha}\| u(s)- v(s)\|_rds\\
%&\lesssim&\|\varphi-\psi\|_q+M^\alpha\int_0^t(t-s)^{-\frac{d}{2}(\frac{\alpha+1}{r}-\frac{1}{q})-\frac{b}{2}}s^{-\beta(\alpha+1)}s^\beta\| u(s)- v(s)\|_rds\\
&\lesssim&\|\varphi-\psi\|_q+M^\alpha d(u,v)\int_0^t(t-s)^{-\frac{d}{2}(\frac{\alpha+1}{r}-\frac{1}{q})-\frac{b}{2}}s^{-\beta(\alpha+1)}ds\\
&\lesssim&\|\varphi-\psi\|_q+M^\alpha d(u,v)t^{1-\frac{d\alpha}{2q}-\frac{b}{2}}\int_0^1(1-\sigma)^{-\frac{d}{2}(\frac{\alpha+1}{r}-\frac{1}{q})-\frac{b}{2}}\sigma^{-\beta(\alpha+1)}d\sigma
\end{eqnarray*}
as $\beta=\frac{d}{2}(\frac{1}{q}-\frac{1}{r})$. It follows from $q>q_c$ that $1-\frac{d\alpha}{2q}-\frac{b}{2}>0$ and from $r>q$ that 
\[
\frac{d}{2}(\frac{\alpha+1}{r}-\frac{1}{q})+\frac{b}{2}<\frac{d}{2}(\frac{\alpha+1}{r}-\frac{1}{r})+\frac{b}{2}=\frac{d\alpha}{2r}+\frac{b}{2}<\frac{d\alpha}{2q}+\frac{b}{2}<1.
\]
This together with $\beta(\alpha+1)<1$ imply that $\int_0^1(1-\sigma)^{-\frac{d}{2}(\frac{\alpha+1}{r}-\frac{1}{q})-\frac{b}{2}}\sigma^{-\beta(\alpha+1)}d\sigma<\infty.$ Hence,
\begin{equation}\label{3}
\|\mathcal{J}_\varphi(u)(t)-\mathcal{J}_\psi(v)(t)\|_q\lesssim\|\varphi-\psi\|_q+M^\alpha T^{1-\frac{d\alpha}{2q}-\frac{b}{2}}d(u,v).
\end{equation}
Similarly,  by Proposition \ref{est0}  with exponent pairs $(\frac{r}{\alpha+1},r)$  we get
\begin{eqnarray*}
&&\|\mathcal{J}_\varphi(u)(t)-\mathcal{J}_\psi(v)(t)\|_r\\
%&\leq&\|e^{-t\la}\varphi-e^{-t\la}\psi\|_r+|\mu|\int_0^t \|e^{-(t-s)\la}(|\cdot|^{-b}[|u(s)|^\alpha u(s)-|v(s)|^\alpha v(s)])\|_rds\\
%&\lesssim&t^{-\frac{d}{2}(\frac{1}{q}-\frac{1}{r})}\|\varphi-\psi\|_q+\int_0^t(t-s)^{-\frac{d}{2}(\frac{\alpha+1}{r}-\frac{1}{r})-\frac{b}{2}}\||u(s)|^\alpha u(s)-|v(s)|^\alpha v(s)\|_{\frac{r}{\alpha+1}}ds\\
%&\lesssim&t^{-\beta}\|\varphi-\psi\|_q+\int_0^t(t-s)^{-\frac{d\alpha}{2r}-\frac{b}{2}}\||u(s)|^\alpha u(s)-|v(s)|^\alpha v(s)\|_{\frac{r}{\alpha+1}}ds\\
&\lesssim&t^{-\beta}\|\varphi-\psi\|_q+\int_0^t(t-s)^{-\frac{d\alpha}{2r}-\frac{b}{2}}(\|u(s)\|_r^\alpha+\|v(s)\|_r^\alpha)\| u(s)- v(s)\|_r ds\\
&\lesssim&t^{-\beta}\|\varphi-\psi\|_q+M^\alpha d(u,v)\int_0^t(t-s)^{-\frac{d\alpha}{2r}-\frac{b}{2}}s^{-\beta(\alpha+1)} ds.\\
\end{eqnarray*}
Hence,
\begin{eqnarray}\label{4}
&&t^\beta\|\mathcal{J}_\varphi(u)(t)-\mathcal{J}_\psi(v)(t)\|_r \nonumber \\
&\lesssim&\|\varphi-\psi\|_q+M^\alpha d(u,v)t^\beta\int_0^t(t-s)^{-\frac{d\alpha}{2r}-\frac{b}{2}}s^{-\beta(\alpha+1)} ds \nonumber \\
&\lesssim&\|\varphi-\psi\|_q+M^\alpha d(u,v)t^{1-\frac{d\alpha}{2q}-\frac{b}{2}}\int_0^1(1-\sigma)^{-\frac{d\alpha}{2r}-\frac{b}{2}}\sigma^{-\beta(\alpha+1)} d\sigma \nonumber \\
& \lesssim & \|\varphi-\psi\|_q+M^\alpha T^{1-\frac{d\alpha}{2q}-\frac{b}{2}}d(u,v).
\end{eqnarray}
By  \eqref{3} and \eqref{4},  we obtain
\begin{equation}\label{5}
d(\mathcal{J}_\varphi(u),\mathcal{J}_\psi(v))\leq c\|\varphi-\psi\|_q+cM^\alpha T^{1-\frac{d\alpha}{2q}-\frac{b}{2}}d(u,v)
\end{equation} for some $c>0$.

For a given $\varphi\in L^q(\rd)$, choose $\rho>0$ so that $\|\varphi||_q\leq\rho$.  Take $M=2c\rho$. Then for $u\in B_M^T$, from \eqref{5} it follows that
\[
d(\mathcal{J}_\varphi(u),0)\leq c\rho+cM^\alpha T^{1-\frac{d\alpha}{2q}-\frac{b}{2}}d(u,0)\leq \frac{M}{2}+cM^{\alpha+1} T^{1-\frac{d\alpha}{2q}-\frac{b}{2}}\leq M
\]
provided $T>0$ small enough so that $cM^{\alpha} T^{1-\frac{d\alpha}{2q}-\frac{b}{2}}\leq\frac{1}{2}$. This is possible as $1-\frac{d\alpha}{2q}-\frac{b}{2}>0$ as a consequence of $q>q_c$ as mentioned earlier. Thus $\mathcal{J}_\varphi(u)\in B_M^T$. Note that $T\sim M^{-\alpha(1-\frac{d\alpha}{2q}-\frac{b}{2})^{-1}}\sim \rho^{-\alpha(1-\frac{d\alpha}{2q}-\frac{b}{2})^{-1}}$ and thus this time $T$ only depends on $\|\varphi\|_q$ rather than  on the profile of $\varphi$ itself.

Also for $u,v\in B_M^T$, from \eqref{5} we have
\begin{align*}
d(\mathcal{J}_\varphi(u),\mathcal{J}_\varphi(v))\leq cM^\alpha T^{1-\frac{d\alpha}{2q}-\frac{b}{2}}d(u,v)\leq\frac{1}{2}d(u,v)
\end{align*} which says $\mathcal{J}_\varphi$ is a contraction in $B_M^T$. Therefore, it has a unique fixed point say $u$ i.e. there is a unique $u\in B_M^T$ such that $\mathcal{J}_\varphi(u)=u$. This completes the proof of part (i).\\

\noindent 
 \textbf{Part (ii) (i.e. uniqueness) with the case  $q>q_c$.} Let $u_1,u_2$ be two solution satisfying 
 \[
 \sup_{t\in[0,T ]} \|u_j(t)\|_q<\infty,\qquad \sup_{t\in[0,T ]} t^{\frac{d}{2}(\frac{1}{q}-\frac{1}{r})}\|u_j(t)\|_r<\infty.
 \] 
 Choose $\tilde{M}$ big enough so that 
\[\sup_{t\in[0,T ]} \|u_j(t)\|_q<\tilde{M},\qquad\sup_{t\in[0,T ]} t^\beta\|u_j(t)\|_r<\tilde{M}.\]
Then $u_1,u_2\in B_{\tilde{M}}^T$. Let $t_0$ be the infimum of $t$ in $[0,T)$ where $u_1(t)\neq u_2(t)$, then as in the above, one can choose a $0<\tilde{T}<T-t_0$ depending on $\|u_1(t_0)\|_q=\|u_2(t_0)\|_q.$ So that $\tilde{\mathcal{J}}_{u_1(t_0)}$ given by
\[
\tilde{\mathcal{J}}_{u_1(t_0)}(u)(t)=e^{-(t-t_0)\la}\varphi+\mu\int_{t_0}^t e^{-(t-s)\la}(|\cdot|^{-b}|u(s)|^\alpha u(s))ds
\]
has a unique fixed point in $B_{\tilde{M}}^{\tilde{T}}$. Since $u_j$ are solutions to \eqref{0} they are also fixed point of $\tilde{\mathcal{J}}_{u_1(t_0)}$ and hence $u_1(t)=t_2(t)$ for $t\in[t_0,t_0+\tilde{T})$ which is a contraction.

%For part (ii) with $q=q_c$,

Let $u,v$ be the unique solutions with data $\varphi,\psi$ respectively, the from \eqref{5}, it follows that
\begin{align*}
d(u,v)\leq c\|\varphi-\psi\|_q+cM^\alpha T^{1-\frac{d\alpha}{2q}-\frac{b}{2}}d(u,v)\leq c\|\varphi-\psi\|_q+\frac{1}{2}d(u,v)\Longrightarrow d(u,v)\leq 2c\|\varphi-\psi\|_q
\end{align*}implying continuous dependency of solution on data.

Blowup alternatives, part \eqref{iii} regarding $T_{max}$  are standard. 
\end{proof}
\subsection{Global Existence}
Using \eqref{2} for $q=q_c$, if    $\frac{2-b}{\tilde{s}_2+2}<\alpha<\frac{2-b}{\tilde{s}_1}$ i.e. $\frac{d}{\tilde{s}_2+2}<q_c<\frac{d}{\tilde{s}_1}$ there exists $r>q_c$ such that 
\begin{equation}\label{6}
 \tilde{s}_1<\frac{d}{r}<\frac{d}{q_c}<b+\frac{d(\alpha+1)}{r}<\tilde{s}_2+2.
\end{equation}
Set \begin{equation}\label{b1}
\beta=\frac{d}{2}(\frac{1}{q_c}-\frac{1}{r})=\frac{2-b}{2\alpha}-\frac{d}{2r}
\end{equation} as before.  Note that using $q_c=\frac{d\alpha}{2-b}$ and \eqref{6},  $$\frac{d(\alpha+1)}{q_c}-2=\frac{d}{q_c}-b<\frac{d(\alpha+1)}{r}$$ which implies $\beta(\alpha+1)<1$. Also $\frac{d\alpha}{2r}+\frac{b}{2}<\frac{d\alpha}{2q_c}+\frac{b}{2}=1.$
\begin{theorem}[global existence]\label{global2}
Let $0 < b < \min(2,d)$ and $\tilde{s}_1<\frac{\tilde{s}_2+2-b}{\alpha+1}$, $\frac{d}{\tilde{s}_2+2}<q_c<\frac{d}{\tilde{s}_1} $. Let $r$  verify \eqref{6} and $\beta$ be as defined in \eqref{b1}. %and    \textcolor{red}{?}. 
Suppose that $\rho > 0$ and $M > 0$ satisfy the inequality\[
c\rho+cM^{\alpha+1} \leq M, 
\] where $c = c(\alpha, d, b, r) > 0$  %\textcolor{red}{($\gamma$ should be $b$)}
  is a constant and can explicitly be computed. Let $\varphi\in\mathcal{S}'(\rd)$ be %a tempered distribution 
  such that
\begin{equation}\label{8}
\sup_{t>0}t^\beta\|e^{-t\la} \varphi\|_r \leq\rho.
\end{equation}
Then there exists a unique global solution u of \eqref{0} such that
\[
\sup_{t>0}t^\beta\|u(t)\|_r \leq M. 
\]
Furthermore,
\begin{enumerate}
\item $u(t)-e^{-t\la}\varphi\in C([0,\infty);L^s(\rd))$, for $\tilde{s}_1<\frac{d}{q_c} <\frac{d}{s}<b +\frac{(\alpha+1)d}{r}<\tilde{s}_2+2$\label{g21}
\item $u(t)-e^{-t\la}\varphi\in L^\infty((0,\infty);L^{q_c}(\rd))$, if $\tilde{s}_1<\frac{d}{q_c}< b +\frac{\alpha+1}{r}<\tilde{s}_2+2$ \label{g22}
\item $\lim_{t\to0} u(t) = \varphi$ in $L^{s}(\rd)$ if $\varphi\in L^{s}(\rd)$ for $s$ satisfying $\frac{d}{q_c}\leq\frac{d}{s} < b +\frac{\alpha+1}{r}<\tilde{s}_2+2$\label{g23}
% and hence in the sense of distributions.
\item $\lim_{t\to0} u(t) = \varphi$ in $\mathcal{S}'(\rd)$ for $\varphi\in\mathcal{S}'(\rd)$ \label{g25}
\item $\sup_{t>0} t^{\frac{2-b}{2\alpha}-\frac{d}{2q}}\|u(t)\|_q \leq C_M<\infty,$ for all $q\in [r,\frac{d}{\tilde{s}_1})$ with $C_M\to0$ as $M\to0$\label{g24}. 
\end{enumerate}

Moreover, if $\varphi$ and $\psi$ satisfy \eqref{6} and if $u$ and $v$ be respectively the solutions of \eqref{0} with initial data $\varphi$ and $\psi$. Then
  \[
\sup_{t>0}t^{\frac{2-b}{2\alpha}-\frac{d}{2q}} \|u(t)-v(t)\|_q \leq C \sup_{t>0}t^\beta\|e^{-t\la}(\varphi-\psi)\|_r , \text{ for all }q\in[r,\frac{d}{\tilde{s}_1}).
\]
If in addition, $e^{-t\la}(\varphi-\psi)$ has the stronger decay property
\[
\sup_{t>0}t^{\beta+\delta}\|e^{-t\la}(\varphi-\psi)\|_r <\infty, 
\]
for some $\delta > 0$ such that $\beta(\alpha + 1) + \delta < 1$, and with $M$ perhaps smaller, then\begin{equation}\label{9} \sup_{t>0}t^{\beta+\delta}\|u(t)-v(t)\|_r \leq C\sup_{t>0}t^{\beta+\delta}\|e^{-t\la}(\varphi-\psi)\|_r
\end{equation}
where $C>0$ is a constant.
\end{theorem}
\begin{proof}[{\bf Proof}]
Let 
$$B_M=\{u:(0,\infty)\to L^r(\rd)): \sup_{t>0} t^\beta\|u(t)\|_r\leq M\}$$ and
\[
d(u,v)=: \sup_{t>0} t^\beta\|u(t)-v(t)\|_r.
\]
Let $\varphi,\psi\in L^q(\rd)$ and $u,v\in B_M.$  Using  Proposition \ref{est0}  with exponent pairs $(\frac{r}{\alpha+1},r)$  we get
\begin{eqnarray}\label{25}
&& t^\beta \|\mathcal{J}_\varphi(u)(t)-\mathcal{J}_\psi(v)(t)\|_r\nonumber\\
%&\leq&\|e^{-t\la}\varphi-e^{-t\la}\psi\|_r+|\mu|\int_0^t \|e^{-(t-s)\la}(|\cdot|^{-b}[|u(s)|^\alpha u(s)-|v(s)|^\alpha v(s)])\|_rds\nonumber\\
%&\lesssim&\|e^{-t\la}(\varphi-\psi)\|_r+\int_0^t(t-s)^{-\frac{d}{2}(\frac{\alpha+1}{r}-\frac{1}{r})-\frac{b}{2}}\||u(s)|^\alpha u(s)-|v(s)|^\alpha v(s)\|_{\frac{r}{\alpha+1}}ds\nonumber\\
%&\lesssim&\|e^{-t\la}(\varphi-\psi)\|_r+\int_0^t(t-s)^{-\frac{d\alpha}{2r}-\frac{b}{2}}\||u(s)|^\alpha u(s)-|v(s)|^\alpha v(s)\|_{\frac{r}{\alpha+1}}ds\nonumber\\
& \lesssim& t^\beta  \|e^{-t\la}(\varphi-\psi)\|_r+ t^\beta \int_0^t(t-s)^{-\frac{d\alpha}{2r}-\frac{b}{2}}(\|u(s)\|_r^\alpha+\|v(s)\|_r^\alpha)\| u(s)- v(s)\|_r ds\\
&\lesssim& t^\beta \|e^{-t\la}(\varphi-\psi)\|_r+ t^\beta M^\alpha d(u,v)\int_0^t(t-s)^{-\frac{d\alpha}{2r}-\frac{b}{2}}s^{-\beta(\alpha+1)} ds\nonumber\\
& \lesssim & t^\beta\|e^{-t\la}(\varphi-\psi)\|_r+M^\alpha d(u,v)t^{\beta+1-\frac{d\alpha}{2r}-\frac{b}{2}-\beta(\alpha+1)}\int_0^1(1-\sigma)^{-\frac{d\alpha}{2r}-\frac{b}{2}}\sigma^{-\beta(\alpha+1)} d\sigma.
\end{eqnarray}
Since $\beta+1-\frac{d\alpha}{2r}-\frac{b}{2}-\beta(\alpha+1)=0$,  we get
\begin{equation}\label{7}
d(\mathcal{J}_\varphi(u),\mathcal{J}_\psi(v))\lesssim\sup_{t>0}t^\beta\|e^{-t\la}(\varphi-\psi)\|_r+M^\alpha d(u,v).
\end{equation}
Putting $\psi,v=0$, and using the hypothesis we have
\[
\sup_{t>0}t^\beta\|\mathcal{J}_\varphi(u)(t)\|_r=d(\mathcal{J}_\varphi(u),0)\leq c\sup_{t>0}t^\beta\|e^{-t\la}\varphi\|_r+cM^\alpha d(u,0)\leq c\rho+cM^\alpha\leq M
\]implying $\mathcal{J}_\varphi(u)\in B_M$. Putting $\psi=\varphi$ in \eqref{7}
\[
d(\mathcal{J}_\varphi(u),\mathcal{J}_\varphi(v))\leq cM^\alpha d(u,v)
\]implying $\mathcal{J}_\varphi(u)$ is a contraction in $B_M$ as $cM^\alpha<1$. Hence it has a unique fixed point in $B_M$.

Now we will prove \eqref{g21}. Lets prove the continuity at $t=0$ first. Take $s$ satisfying \[
 \tilde{s}_1<\frac{d}{q_c}\leq\frac{d}{s}<b+\frac{d(\alpha+1)}{r}<\tilde{s}_2+2.
\]
Using Proposition \ref{est0} for the pair $(\frac{r}{\alpha+1},s)$ we have
\begin{eqnarray*}
\|u(t)-e^{-t\la}\varphi\|_s
%&\lesssim&\int_0^t \|e^{-(t-\tau)\la}(|\cdot|^{-b}|u(\tau)|^\alpha u(\tau))\|_sd\tau\\
%&\lesssim&\int_0^t(t-\tau)^{-\frac{d}{2}(\frac{\alpha+1}{r}-\frac{1}{s})-\frac{b}{2}}\||u(\tau)|^\alpha u(\tau)\|_{\frac{r}{\alpha+1}}d\tau\\
&\lesssim &\int_0^t(t-\tau)^{-\frac{d}{2}(\frac{\alpha+1}{r}-\frac{1}{s})-\frac{b}{2}}\|u(\tau)\|_r^{\alpha+1}d\tau\\
&\leq&M^{\alpha+1}\int_0^t(t-\tau)^{-\frac{d}{2}(\frac{\alpha+1}{r}-\frac{1}{s})-\frac{b}{2}}\tau^{-\beta(\alpha+1)}d\tau\\
&=&M^{\alpha+1}t^{1-\frac{d}{2}(\frac{\alpha+1}{r}-\frac{1}{s})-\frac{b}{2}-\beta(\alpha+1)}\int_0^1(1-\sigma)^{-\frac{d}{2}(\frac{\alpha+1}{r}-\frac{1}{s})-\frac{b}{2}}\sigma^{-\beta(\alpha+1)}d\sigma\\
&=&M^{\alpha+1}t^{\frac{d}{2s}-\frac{2-b}{2\alpha}}\int_0^1(1-\sigma)^{-\frac{d}{2}(\frac{\alpha+1}{r}-\frac{1}{s})-\frac{b}{2}}\sigma^{-\beta(\alpha+1)}d\sigma.
\end{eqnarray*}
Note that 
\[
\frac{d}{2}(\frac{\alpha+1}{r}-\frac{1}{s})+\frac{b}{2}<\frac{d}{2}(\frac{\alpha+1}{r}-\frac{1}{q_c})+\frac{b}{2}<\frac{d}{2}(\frac{\alpha+1}{r}-\frac{1}{r})+\frac{b}{2}=\frac{d\alpha}{2r}+\frac{b}{2}<\frac{d\alpha}{2q_c}+\frac{b}{2}=1
\]and $\beta(\alpha+1)<1$ implying $\int_0^1(1-\sigma)^{-\frac{d}{2}(\frac{\alpha+1}{r}-\frac{1}{s})-\frac{b}{2}}\sigma^{-\beta(\alpha+1)}d\sigma<\infty$. Thus
\begin{equation}\label{p5}
\|u(t)-e^{-t\la}\varphi\|_s\lesssim M^{\alpha+1}t^{\frac{d}{2s}-\frac{2-b}{2\alpha}}.
\end{equation}
%This proves \eqref{g22}.???  \textcolor{red}{How?}
Since $\frac{d}{2s}-\frac{2-b}{2\alpha}>0$ (this follows fro the condition $\frac{d}{q_c}<\frac{d}{s}$ in \eqref{g21}), we get \begin{equation}\label{p6}
u(t)-e^{-t\la}\varphi\to0
\end{equation} in $L^s(\rd)$ as $t\to0+$ and hence continuity at $t=0$.  %in particular $s$ can be chosen $q_c$   and hence in $L^{q_c}(\rd)$

For other $t>0$, first note that $u(t)-e^{-t\la}\varphi-u(t_0)+e^{-t_0\la}\varphi=\mu\int_0^t e^{-(t-s)\la}(|\cdot|^{-b}|u(s)|^\alpha u(s))ds-\mu\int_0^{t_0} e^{-(t_0-s)\la}(|\cdot|^{-b}|u(s)|^\alpha u(s))ds$, but with $0<t<t_0$
\begin{align}
&\|\int_0^t e^{-(t-\tau)\la}(|\cdot|^{-b}|u(\tau)|^\alpha u(\tau))d\tau-\int_0^t e^{-(t_0-\tau)\la}(|\cdot|^{-b}|u(\tau)|^\alpha u(\tau))d\tau\|_s\nonumber\\
&\leq\int_0^{t_0}\chi_{[0,t]}(\tau) \|e^{-(t-\tau)\la}(|\cdot|^{-b}|u(\tau)|^\alpha u(\tau))- e^{-(t_0-\tau)\la}(|\cdot|^{-b}|u(\tau)|^\alpha u(\tau))\|_sd\tau\to 0\label{p3}
\end{align}as $t\uparrow t_0$, using dominated convergence (in fact note that for each $\tau$, the integrand say $I(\tau)$ goes to zero as $t\uparrow t_0$ and now 
\begin{align*}
I(\tau)&\leq\chi_{[0,t]}(\tau)\left[ \|e^{-(t-\tau)\la}(|\cdot|^{-b}|u(\tau)|^\alpha u(\tau))\|_s+\| e^{-(t_0-\tau)\la}(|\cdot|^{-b}|u(\tau)|^\alpha u(\tau))\|_s\right]\\
&\lesssim\chi_{[0,t]}(\tau)\left[(t-\tau)^{-\frac{d}{2}(\frac{\alpha+1}{r}-\frac{1}{s})-\frac{b}{2}}+(t_0-\tau)^{-\frac{d}{2}(\frac{\alpha+1}{r}-\frac{1}{s})-\frac{b}{2}}\right]\|u(\tau)\|_s^{\alpha+1}\\
&\leq M^{\alpha+1}\chi_{[0,t)}\left[(t-\tau)^{-\frac{d}{2}(\frac{\alpha+1}{r}-\frac{1}{s})-\frac{b}{2}}+(t_0-\tau)^{-\frac{d}{2}(\frac{\alpha+1}{r}-\frac{1}{s})-\frac{b}{2}}\right]\tau^{-\beta(\alpha+1)}
\end{align*} which is integrable). 
On the other hand for $\frac{t_0}{2}\leq t<t_0$,
\begin{align*}
&\|\int_0^t e^{-(t_0-\tau)\la}(|\cdot|^{-b}|u(\tau)|^\alpha u(\tau))d\tau-\int_0^{t_0} e^{-(t_0-\tau)\la}(|\cdot|^{-b}|u(\tau)|^\alpha u(\tau))d\tau\|_s\\
&=\|\int_t^{t_0} e^{-(t_0-\tau)\la}(|\cdot|^{-b}|u(\tau)|^\alpha u(\tau))d\tau\|_s\\
&\leq M^{\alpha+1}\int_t^{t_0}(t_0-\tau)^{-\frac{d}{2}(\frac{\alpha+1}{r}-\frac{1}{s})-\frac{b}{2}}\tau^{-\beta(\alpha+1)}d\tau\\
&=M^{\alpha+1}t_0^{1-\frac{d}{2}(\frac{\alpha+1}{r}-\frac{1}{s})-\frac{b}{2}-\beta(\alpha+1)}\int_{t/t_0}^1(1-\sigma)^{-\frac{d}{2}(\frac{\alpha+1}{r}-\frac{1}{s})-\frac{b}{2}}\sigma^{-\beta(\alpha+1)}d\sigma\\
&\leq M^{\alpha+1}2^{2\beta(\alpha+1)}t_0^{1-\frac{d}{2}(\frac{\alpha+1}{r}-\frac{1}{s})-\frac{b}{2}-\beta(\alpha+1)}\int_{t/t_0}^1(1-\sigma)^{-\frac{d}{2}(\frac{\alpha+1}{r}-\frac{1}{s})-\frac{b}{2}}d\sigma\\
&= M^{\alpha+1}2^{2\beta(\alpha+1)}t_0^{1-\frac{d}{2}(\frac{\alpha+1}{r}-\frac{1}{s})-\frac{b}{2}-\beta(\alpha+1)}\int_0^{1-t/t_0}\sigma^{-\frac{d}{2}(\frac{\alpha+1}{r}-\frac{1}{s})-\frac{b}{2}}d\sigma\\
&= M^{\alpha+1}2^{2\beta(\alpha+1)}t_0^{1-\frac{d}{2}(\frac{\alpha+1}{r}-\frac{1}{s})-\frac{b}{2}-\beta(\alpha+1)}\frac{1}{1-\frac{d}{2}(\frac{\alpha+1}{r}-\frac{1}{s})-\frac{b}{2}}(1-\frac{t}{t_0})^{1-\frac{d}{2}(\frac{\alpha+1}{r}-\frac{1}{s})-\frac{b}{2}}
\to0
\end{align*}as $t\uparrow t_0$. Using this and \eqref{p3} we get $u(t)-e^{-t_0\la}\varphi\to u(t_0)-e^{-t_0\la}\varphi$ as $t\uparrow t_0$ and so left continuity is established. Now for $0<t_0<t$,
\begin{align}
&\|\int_0^t e^{-(t-\tau)\la}(|\cdot|^{-b}|u(\tau)|^\alpha u(\tau))d\tau-\int_0^{t_0} e^{-(t-\tau)\la}(|\cdot|^{-b}|u(\tau)|^\alpha u(\tau))d\tau\|_s\nonumber\\
&=\|\int_{t_0}^t e^{-(t-\tau)\la}(|\cdot|^{-b}|u(\tau)|^\alpha u(\tau))d\tau\|_s\nonumber\\
&\leq\int_{t_0}^t(t-\tau)^{-\frac{d}{2}(\frac{\alpha+1}{r}-\frac{1}{s})-\frac{b}{2}}\tau^{-\beta(\alpha+1)}d\tau\nonumber\\
&=t^{1-\frac{d}{2}(\frac{\alpha+1}{r}-\frac{1}{s})-\frac{b}{2}-\beta(\alpha+1)}\int_{t_0/t}^1(1-\sigma)^{-\frac{d}{2}(\frac{\alpha+1}{r}-\frac{1}{s})-\frac{b}{2}}\sigma^{-\beta(\alpha+1)}d\sigma\to0\label{p4}
\end{align}as $t\downarrow t_0$. And arguing as \eqref{p3}
\begin{align*}
&\|\int_0^{t_0} e^{-(t-\tau)\la}(|\cdot|^{-b}|u(\tau)|^\alpha u(\tau))d\tau-\int_0^{t_0} e^{-(t_0-\tau)\la}(|\cdot|^{-b}|u(\tau)|^\alpha u(\tau))d\tau\|_s\\
&\leq\int_0^{t_0}\| e^{-(t-\tau)\la}(|\cdot|^{-b}|u(\tau)|^\alpha u(\tau))- e^{-(t_0-\tau)\la}(|\cdot|^{-b}|u(\tau)|^\alpha u(\tau))\|_sd\tau\to0
\end{align*}as $t\downarrow t_0$. This together with \eqref{p4} implies $u(t)-e^{-t_0\la}\varphi\to u(t_0)-e^{-t_0\la}\varphi$ as $t\downarrow t_0$ so the right continuity is established. This proves \eqref{g21}. 

\eqref{g22} follows from \eqref{p5}.

  If $\varphi\in L^{s}(\rd)$, then $e^{-t\la}\varphi\to\varphi$ in $L^s(\rd)$  as $t\to0+$(semigroup property). %we also have $e^{-t\la}\varphi\to\varphi$ in $L^s(\rd)$. 
 Thus, using \eqref{p6}, we conclude $u(t) \to\varphi$ in $L^s(\rd)$. %and hence in $\mathcal{S}'(\rd)$. 
This proves \eqref{g23} and hence \eqref{g25}. 

For part  \eqref{g24} we use iteration. Set $r_0=r$, $M_0=M$ we note that we have
\[
\sup_{t>0}t^{\frac{2-b}{2\alpha}-\frac{d}{2r_0}}\|u(t)\|_{r_0}\leq M_0.
\]
%then  $r_1\geq r_0$ satisfies $\tilde{s}_1 <\frac{d}{r_1}<\frac{d}{r_0}  <b+ \frac{d(\alpha+1)}{r_0} < \tilde{s}_2+2$, 
Take
\begin{equation}\label{24}
\frac{1}{r_1}=\begin{cases}
\frac{1}{r_0}-\frac{1}{2}(\frac{2-b}{d}-\frac{\alpha}{r_0})&\text{ if }\frac{1}{r_0}-\frac{1}{2}(\frac{2-b}{d}-\frac{\alpha}{r_0})>\frac{\tilde{s}_1}{d}\\
\text{any number in }(\frac{\tilde{s}_1}{d},\frac{1}{r_0})&\text{ otherwise}.
\end{cases}
\end{equation}
%If $\frac{1}{r_0}-\frac{1}{2}(\frac{2}{d}-\frac{b}{d}-\frac{\alpha}{r_0})>\frac{\tilde{s}_1}{d}$ choose arbitrary $r_1\in[r_0,\frac{d}{\tilde{s}_1})$. 
%Choose $r_1$ such that $\frac{1}{r_0}-\frac{1}{r_1}=\frac{1}{2}(\frac{2}{d}-\frac{b}{d}-\frac{\alpha}{r_0})$  provided this \begin{equation}\label{24}
%r_1<\frac{d}{\tilde{s}_1}.
%\end{equation} 
Then it is easy to see $\tilde{s}_1 <\frac{d}{r_1}<\frac{d}{r_0}$ as $\frac{2-b}{d}-\frac{\alpha}{r_0}>0\Leftrightarrow r_0>q_c$. Thus using \eqref{6} we have \begin{equation}\label{27}
\tilde{s}_1 <\frac{d}{r_1}<\frac{d}{r_0}  <b+ \frac{d(\alpha+1)}{r_0} < \tilde{s}_2+2
\end{equation}
If the first case occurs in \eqref{24}, then $\frac{d}{2}\big(\frac{1}{r_0}-\frac{1}{r_1}\big)=\frac{1}{2}(\frac{2-b}{2}-\frac{d\alpha}{2r_0})<\frac{2-b}{2}-\frac{d\alpha}{2r_0}\Leftrightarrow\frac{d}{2}\big(\frac{\alpha+1}{r_0}-\frac{1}{r_1}\big)<1-\frac{b}{2}$.
If the second case occurs in \eqref{24}, then $\frac{1}{r_0}-\frac{1}{2}(\frac{2-b}{d}-\frac{\alpha}{r_0})<\frac{\tilde{s}_1}{d}<\frac{1}{r_1}$. Therefore
$\frac{d}{2}\big(\frac{1}{r_0}-\frac{1}{r_1}\big)<\frac{d}{2}\big(\frac{1}{r_0}-\frac{1}{r_0}+\frac{1}{2}(\frac{2-b}{d}-\frac{\alpha}{r_0}\big)=\frac{1}{2}(\frac{2-b}{2}-\frac{d\alpha}{2r_0})<\frac{2-b}{2}-\frac{d\alpha}{2r_0}\Rightarrow\frac{d}{2}\big(\frac{\alpha+1}{r_0}-\frac{1}{r_1}\big)<1-\frac{b}{2}$. Thus in both case in \eqref{24}
\[
\frac{d}{2}\big(\frac{\alpha+1}{r_0}-\frac{1}{r_1}\big)<1-\frac{b}{2}.
\]
  %$\tilde{s}_1 <\frac{d}{r_1}<\frac{d}{r_0}  <b+ \frac{d(\alpha+1)}{r_0} < \tilde{s}_2+2$, $\frac{d}{2}\big(\frac{\alpha+1}{r_0}-\frac{1}{r_1}\big)<1-\frac{b}{2}$.
 In view of this and \eqref{27}, using Lemma \ref{23} with $(s,q)=(r_0,r_1)$ we have
   \begin{equation*}\label{}
\sup_{t>0}t^{\frac{2-b}{2\alpha}-\frac{d}{2r_1}}\|u(t)\|_{r_1}\leq cM_0(1+M_0^\alpha)=:M_1
\end{equation*} 
% If \eqref{} fails take arbitrary $r_1\in[r_0,\frac{d}{\tilde{s}_1})$ and use Lemma \ref{23} to conclude \eqref{25} one stop the iteration. %Proceeding like this
If the second case occors in \eqref{24}, we stop the iteration, otherwise will next choose $r_2$ from $r_1$ as we have chosen $r_1$ from $r_0$ above.

Note that this iteration must stop (the second case must occur) at finite steps.
If not we would find $\tilde{s}_1<\cdots<\frac{d}{r_{i+1}}<\frac{d}{r_i}<\cdots<\frac{d}{r_0}$ with
\begin{equation}\label{26}
\sup_{t>0}t^{\frac{2-b}{2\alpha}-\frac{d}{2r_{i+1}}}\|u(t)\|_{r_{i+1}}\leq cM_i(1+M_i^\alpha)=:M_{i+1}
\end{equation}
 and $\frac{1}{r_i}-\frac{1}{r_{i+1}}=\frac{1}{2}(\frac{2-b}{d}-\frac{\alpha}{r_i})\geq\frac{1}{2}(\frac{2-b}{d}-\frac{\alpha}{r_0})>0$. Then $\frac{1}{r_0}=(\frac{1}{r_0}-\frac{1}{r_1})+(\frac{1}{r_1}-\frac{1}{r_2})+\cdots+(\frac{1}{r_i}-\frac{1}{r_{i+1}})+\frac{1}{r_{i+1}}\geq\frac{i+1}{2}(\frac{2-b}{d}-\frac{\alpha}{r_0})\to\infty$ which is a contradiction as $\frac{1}{r_0}<\infty$.

For $q\in(r_i,r_{i+1})$ we use interpolation inequality with \eqref{26}
$\|u(t)\|_q\leq\|u(t)\|_{r_i}^\theta\|u(t)\|_{r_{i+1}}^{1-\theta}$ where $\frac{1}{q}=\frac{\theta}{r_1}+\frac{1-\theta}{r_{i+1}}$.

Further part: Using $\mathcal{J}_\varphi(u)=u$, $\mathcal{J}_\psi(v)=v$ in \eqref{25} we have
\begin{eqnarray*}
\|u(t)-v(t)\|_r&\lesssim&\|e^{-t\la}(\varphi-\psi)\|_r+\int_0^t(t-s)^{-\frac{d\alpha}{2r}-\frac{b}{2}}(\|u(s)\|_r^\alpha+\|v(s)\|_r^\alpha)\| u(s)- v(s)\|_r ds\\
&\leq&\|e^{-t\la}(\varphi-\psi)\|_r+M^\alpha\int_0^t(t-s)^{-\frac{d\alpha}{2r}-\frac{b}{2}}s^{-\beta\alpha}\| u(s)- v(s)\|_r ds.
\end{eqnarray*}
Therefore
\begin{eqnarray*}
t^{\beta+\delta}\|u(t)-v(t)\|_r&\leq&t^{\beta+\delta}\|e^{-t\la}(\varphi-\psi)\|_r+M^\alpha t^{\beta+\delta}\\
&&\cdot\int_0^t(t-s)^{-\frac{d\alpha}{2r}-\frac{b}{2}}s^{-\beta\alpha-\beta-\delta}s^{\beta+\delta}\| u(s)- v(s)\|_r ds\\
&\leq&t^{\beta+\delta}\|e^{-t\la}(\varphi-\psi)\|_r+M^\alpha t^{\beta+\delta}\sup_{\tau>0} \tau^{\beta+\delta}\| u(\tau)- v(\tau)\|_r\\
&&\cdot\int_0^t(t-s)^{-\frac{d\alpha}{2r}-\frac{b}{2}}s^{-\beta(\alpha+1)-\delta} ds\\
&\leq&t^{\beta+\delta}\|e^{-t\la}(\varphi-\psi)\|_r+M^\alpha t^{\beta+\delta -\frac{d\alpha}{2r}-\frac{b}{2}-\beta(\alpha+1)-\delta+1}\\
&&\cdot\sup_{\tau>0} \tau^{\beta+\delta}\| u(\tau)- v(\tau)\|_r\int_0^1(1-\sigma)^{-\frac{d\alpha}{2r}-\frac{b}{2}}\sigma^{-\beta(\alpha+1)-\delta} d\sigma
\end{eqnarray*}
Note that $\int_0^1(1-\sigma)^{-\frac{d\alpha}{2r}-\frac{b}{2}}\sigma^{-\beta(\alpha+1)-\delta} d\sigma<\infty$ as $\frac{d\alpha}{2r}+\frac{b}{2}<1$, $\beta(\alpha+1)+\delta<1$. Thus
\begin{equation*}
\sup_{t>0}t^{\beta+\delta}\|u(t)-v(t)\|_r\leq ct^{\beta+\delta}\|e^{-t\la}(\varphi-\psi)\|_r+cM^\alpha\sup_{\tau>0} \tau^{\beta+\delta}\| u(\tau)- v(\tau)\|_r
\end{equation*}
choosing $M>0$ so that $cM^\alpha<\frac{1}{2}$ we achieve \eqref{9}.

Moreover part: Note that  
\[
u(t)-v(t)=e^{-t\la/2}(u(t/2)-v(t/2))+\mu\int_\frac{t}{2}^t e^{-(t-s)\la}(|\cdot|^{-b}[|u(s)|^\alpha u(s)-|v(s)|^\alpha v(s)])ds.
\]
Therefore for $q\in[r,\frac{d}{\tilde{s}_1})$   using \eqref{6}, one has \[\tilde{s}_1<\frac{d}{q}<\frac{d}{r} <\tilde{s}_2+2\qquad\text{and}\qquad \tilde{s}_1<\frac{d}{q}<b+\frac{d(\alpha+1)}{q}<\tilde{s}_2+2.\]
Then using Theorem \ref{ste} for $(r,q)$ and Proposition \ref{est0} for $(\frac{q}{\alpha+1},q)$ we have
\begin{eqnarray*}
&&\|u(t)-v(t)\|_q\\
%&\lesssim&\|e^{-t\la/2}(u(t/2)-v(t/2))\|_q+\int_\frac{t}{2}^t \|e^{-(t-s)\la}(|\cdot|^{-b}[|u(s)|^\alpha u(s)-|v(s)|^\alpha v(s)])\|_qds\\
%&\lesssim&t^{-\frac{d}{2}(\frac{1}{r}-\frac{1}{q})}\|u(t/2)-v(t/2)\|_r+\int_\frac{t}{2}^t (t-s)^{-\frac{d}{2}(\frac{\alpha+1}{q}-\frac{1}{q})-\frac{b}{2}}\||u(s)|^\alpha u(s)-|v(s)|^\alpha v(s)\|_{\frac{q}{\alpha+1}}ds\\
%&\lesssim&t^{-\frac{d}{2}(\frac{1}{r}-\frac{1}{q})}\|u(t/2)-v(t/2)\|_r+\int_\frac{t}{2}^t (t-s)^{-\frac{d\alpha}{2q}-\frac{b}{2}}\||u(s)|^\alpha u(s)-|v(s)|^\alpha v(s)\|_{\frac{q}{\alpha+1}}ds\\
&\lesssim&t^{-\frac{d}{2}(\frac{1}{r}-\frac{1}{q})}\|u(t/2)-v(t/2)\|_r+\int_\frac{t}{2}^t (t-s)^{-\frac{d\alpha}{2q}-\frac{b}{2}}(\| u(s)\|_q^\alpha+\|v(s)\|_q^\alpha)\| u(s)- v(s)\|_qds\\
&\lesssim&t^{-\frac{d}{2}(\frac{1}{r}-\frac{1}{q})}\|u(t/2)-v(t/2)\|_r+C_M^\alpha\int_\frac{t}{2}^t (t-s)^{-\frac{d\alpha}{2q}-\frac{b}{2}}s^{-\alpha\beta(q)}\| u(s)- v(s)\|_qds,
\end{eqnarray*}
using \eqref{g24},  $\beta(q):=\frac{2-b}{2\alpha}-\frac{d}{2q}$.

Let $\beta(q)=\frac{2-b}{2\alpha}-\frac{d}{2q}$. Then
\begin{eqnarray*}
t^{\beta(q)}\|u(t)-v(t)\|_q
%&\lesssim&t^{\beta(q)-\frac{d}{2}(\frac{1}{r}-\frac{1}{q})}\|u(t/2)-v(t/2)\|_r\\
%&&+C_M^\alpha t^{\beta(q)}\int_\frac{t}{2}^t (t-s)^{-\frac{d\alpha}{2q}-\frac{b}{2}}s^{-(\alpha+1)\beta(q)}s^{\beta(q)}\| u(s)- v(s)\|_qds\\
&\lesssim&t^\beta\|u(t/2)-v(t/2)\|_r\\
&&+C_M^\alpha t^{\beta(q)}\sup_{\tau>0} \tau^{\beta(q)}\| u(\tau)- v(\tau)\|_q\int_\frac{t}{2}^t (t-s)^{-\frac{d\alpha}{2q}-\frac{b}{2}}s^{-(\alpha+1)\beta(q)}ds\\
&\lesssim&(t/2)^\beta\|u(t/2)-v(t/2)\|_r+C_M^\alpha t^{\beta(q)-\frac{d\alpha}{2q}-\frac{b}{2}-(\alpha+1)\beta(q)+1}\\
&&\cdot\sup_{\tau>0} \tau^{\beta(q)}\| u(\tau)- v(\tau)\|_q\int_\frac{1}{2}^1 (1-\sigma)^{-\frac{d\alpha}{2q}-\frac{b}{2}}\sigma^{-(\alpha+1)\beta(q)}d\sigma.
\end{eqnarray*}
Since $\beta(q)-\frac{d\alpha}{2q}-\frac{b}{2}-(\alpha+1)\beta(q)+1=0$ and $\frac{d\alpha}{2q}+\frac{b}{2}<\frac{d\alpha}{2r}+\frac{b}{2}<1$, we have
\begin{align*}
t^{\beta(q)}\|u(t)-v(t)\|_q\leq c\sup_{\tau>0}\tau^\beta\|u(\tau)-v(\tau)\|_r+cC_M^\alpha\sup_{\tau>0} \tau^{\beta(q)}\| u(\tau)- v(\tau)\|_q
\end{align*}
Choosing $M>0$ small enough so that $cC_M<\frac{1}{2}$, we get 
\[
\sup_{t>0}t^{\beta(q)}\|u(t)-v(t)\|_q\lesssim\sup_{t>0}t^\beta\|u(t)-v(t)\|_r\lesssim\sup_{t>0}t^{\beta}\|e^{-t\la}(\varphi-\psi)\|_r
\]
by using \eqref{9} with $\delta=0$.
This completes the proof.
\end{proof}

\begin{lemma}[A priori estimate]\label{23}
Suppose $ s < q  $, 
\[
\tilde{s}_1 <\frac{d}{q}  <b+ \frac{d(\alpha+1)}{s} < \tilde{s}_2+2,\qquad\frac{d}{2}\big(\frac{\alpha+1}{s}-\frac{1}{q}\big)<1-\frac{b}{2}.
\]
Assume $u$ be solution to \eqref{0} %with data $\varphi\in \mathcal{S}'(\rd)$ 
satisfying 
\[
\sup_{t>0}t^{\frac{2-b}{2\alpha}-\frac{d}{2s}}\|u(t)\|_s\leq A<\infty
\]
then
\[
\sup_{t>0}t^{\frac{2-b}{2\alpha}-\frac{d}{2q}}\|u(t)\|_q\lesssim A(1+A^\alpha)=:C_A<\infty
\]
with $C_A\to0$ as $A\to0$.
\end{lemma}
\begin{proof}[{\bf Proof}]
Note that
\[
u(t)=e^{-\frac{t}{2}\la}u(t/2)+\mu\int_{t/2}^te^{-(t-\sigma )\la}(|\cdot|^{-b} (|u(\sigma)|^\alpha u(\sigma))d\sigma.
\] 
 Using Theorem \ref{ste} for $(p,q)=(s,q)$ and Proposition \ref{est0} for $(p,q)=(\frac{s}{\alpha+1},q)$ we have 
\begin{eqnarray*}
\|u(t)\|_q
%&\lesssim&\|e^{-\frac{t}{2}\la}u(t/2)\|_q+\int_{t/2}^t\|e^{-(t-\tau)\la}(|\cdot|^{-b} (|u(\tau)|^\alpha u(\tau))\|_q d\tau\\
&\lesssim&t^{-\frac{d}{2}(\frac{1}{s}-\frac{1}{q})}\|u(t/2)\|_s+\int_{t/2}^t(t-\tau)^{-\frac{d}{2}(\frac{\alpha+1}{s}-\frac{1}{q})-\frac{b}{2}}\||u(\tau)|^\alpha u(\tau)\|_{\frac{s}{\alpha+1}}d\tau\\
&\lesssim&t^{\frac{d}{2q}-\frac{2-b}{2\alpha}}A+\int_{t/2}^t(t-\tau)^{-\frac{d}{2}(\frac{\alpha+1}{s}-\frac{1}{q})}\| u(\tau)\|_s^{\alpha+1}d\tau\\
&\lesssim&t^{\frac{d}{2q}-\frac{2-b}{2\alpha}}A+A^{\alpha+1}\int_{t/2}^t(t-\tau)^{-\frac{d}{2}(\frac{\alpha+1}{s}-\frac{1}{q})-\frac{b}{2}}\tau^{-\beta(s)(\alpha+1)}d\tau\\
&\lesssim&t^{\frac{d}{2q}-\frac{2-b}{2\alpha}}A+A^{\alpha+1}t^{1-\frac{d}{2}(\frac{\alpha+1}{s}-\frac{1}{q})-\frac{b}{2}-\beta(s)(\alpha+1)}\int_{1/2}^1(1-\sigma)^{-\frac{d}{2}(\frac{\alpha+1}{s}-\frac{1}{q})-\frac{b}{2}}\sigma^{-\beta(s)(\alpha+1)}d\sigma\\
&\lesssim&t^{\frac{d}{2q}-\frac{2-b}{2\alpha}}A+A^{\alpha+1}t^{\frac{d}{2q}-\frac{2-b}{2\alpha}}=A(1+A^\alpha)t^{\frac{d}{2q}-\frac{2-b}{2\alpha}}
\end{eqnarray*} as $\int_{1/2}^1(1-\sigma)^{-\frac{d}{2}(\frac{\alpha+1}{s}-\frac{1}{q})-\frac{b}{2}}\sigma^{-\beta(s)(\alpha+1)}d\sigma<\infty$.
This completes the proof.
\end{proof}
%\begin{corollary}
%\textcolor{red}{please add it,  if it is needed to prove main results in the introduction}
%\end{corollary}
\begin{proof}[{\bf Proof of Theorem \ref{local}  \eqref{ii}}] {\bf The case $q=q_c$.} First note that if the condition \eqref{8} of Theorem \ref{global2} is satisfied on $[0,T)$ (in place of $(0,\infty)$ i.e. if
\begin{equation}\label{8a}
\sup_{t\in[0,T)}t^\beta\|e^{-t\la} \varphi\|_r \leq\rho.
\end{equation}
happens, then the conclusion holds with $(0,\infty)$ replaced by $(0,T)$ i.e. there exists a unique solution $u$ on $[0,T)$ such that \begin{equation*}
\sup_{t\in[0,T)}t^\beta\|u(t)\|_r \leq M.
\end{equation*}

In view of this, it is enough to prove \eqref{8a} with a $T>0$ for a given $\varphi\in L^{q_c}(\rd)$. For $\epsilon>0$,  there exists $\psi\in L^{q_c}(\rd)\cap L^r(\rd)$ such that $\|\varphi-\psi\|_{q_c}<\epsilon$ (by density). Then for $0\leq t<T$, by Theorem \ref{ste}
\begin{align*}
t^\beta\|e^{-t\la} \varphi\|_r &\leq t^\beta\|e^{-t\la}( \varphi-\psi)\|_r +t^\beta\|e^{-t\la} \psi\|_r \\
&\leq ct^\beta t^{-\beta}\| \varphi-\psi\|_{q_c} +ct^\beta\| \psi\|_r\\
&\leq c\epsilon+T^\beta\| \psi\|_r\leq 2c\epsilon
\end{align*} by choosing $T > 0$ small enough depending on $\epsilon$, $\|\psi\|_r$ (which again depends on  $\epsilon$). Now take $\epsilon=\frac{\rho}{2c}$, to achieve \eqref{8a}. This completes the proof. % $\varphi$ \textcolor{red}{satisfies \eqref{8}, on $(0,T)$.???}
\end{proof}
\begin{proof}[{\bf Proof of Theorem \ref{global}}]
For part \eqref{global-1}, using Theorem \ref{ste} note that
\[
t^\beta\|e^{-t\la} \varphi\|_r\leq ct^\beta t^{-\frac{d}{2}(\frac{1}{q_c}-\frac{1}{r})}\|\varphi\|_{q_c}=c\|\varphi\|_{q_c}
\]for all $t>0$ therefore it satisfies the hypothesis \eqref{8} of Theorem \ref{global2} for $\|\varphi\|_{q_c}$ small enough. Hence the result follows from Theorem \ref{global2}.

For part \eqref{global-2}, with the given condition on $\sigma$, the data satisfies the condition in part \eqref{global-1}.

For part \eqref{global-3}, write $|\cdot|^{-\frac{2-b}{\alpha}}=\varphi_1+\varphi_2$ with $\varphi_1=|\cdot|^{-\frac{2-b}{\alpha}}\chi_{\{|x|\leq1\}}$ and $\varphi_2=|\cdot|^{-\frac{2-b}{\alpha}}\chi_{\{|x|>1\}}$, then $\varphi_1\in L^s(\rd)$ for $1\leq s<\frac{d\alpha}{2-b}=q_c$ and $\varphi_2\in L^\sigma(\rd)$ for $q_c<\sigma\leq\infty$.  Then using Theorem \ref{ste}, and the fact that $\tilde{s}_1<\frac{d}{r}<\frac{d}{q_c}<\tilde{s}_2+2$ (see \eqref{6}), $e^{-\la}\varphi_j\in L^r(\rd)$ hence $e^{-\la}|\cdot|^{-\frac{2-b}{\alpha}}\in L^r(\rd)$. By homogeneity of $|\cdot|^{-\frac{2-b}{\alpha}}$ it follows that $\sup_{t>0}t^\beta\|e^{-t\la}|\cdot|^{-\frac{2-b}{\alpha}}\|_r<\infty$. Using positivity of $e^{-t\la}$, from $|\varphi|\leq c|\cdot|^{-\frac{2-b}{\alpha}}$ it follows that, $\sup_{t>0}t^\beta\|e^{-t\la}\varphi\|_r<\infty$. Choosing $c>0$ small enough, $\varphi$ satisfies condition \eqref{8} of Theorem \ref{global2}.
\end{proof}

\subsection{Selfsimilar Solution}
\begin{proof}[{\bf Proof of Theorem \ref{selfsimilar}}]
Let $r$ as in \eqref{6}. Since $|\varphi|\leq\|\omega\|_\infty|\cdot|^{-\frac{2-b}{\alpha}}$, proceeding as in the proof of Theorem \ref{global} \eqref{global-3} above, we achieve  $\varphi$ satisfies \eqref{8}. Then \eqref{0} has a unique global solution $u$ by Theorem \ref{global2} satisfying \eqref{9}.

Since $\varphi$ is homogeneous of degree $-\frac{2-b}{d}$ it follows that $\varphi_\lambda=\varphi$ for all $\lambda$ where $\varphi_\lambda(x) = \lambda^{\frac{2-b}{\alpha}} \varphi(\lambda x)$. Therefore $\varphi_\lambda$ also satisfies \eqref{8} and has a unique global solution $\tilde{u}_\lambda$ by Theorem \ref{global2} satisfying $\sup_{t>0}t^\beta\|\tilde{u}_\lambda(t)\|_r \leq M$. By computation one has $\tilde{u}_\lambda=u_\lambda$, where $u_\lambda$ as defined in \eqref{scl}.
Since $\varphi_\lambda=\varphi$, by uniqueness of solution, we have $u_\lambda=u$.
\end{proof}
\section{Asymptotic Behaviour}\label{AB}
\subsection{Case I: Nonlinear Behaviour}\label{AB1}
\begin{proof}[{\bf Proof of Theorem \ref{asym} \eqref{asym1}}]
Set $\beta(q)=\frac{2-b}{2\alpha}-\frac{d}{2q}$ and $\psi(x)=\omega(x)|x|^{-\frac{2-b}{\alpha}}$. 
Note that $|\varphi(x)-\psi(x)|=0$ for $|x|\geq A$ and $|\varphi(x)-\psi(x)|\leq (c+\|\omega\|_\infty)|x|^{-\frac{2-b}{\alpha}}.$  Hence, 
\[
|\varphi-\psi|\leq  (c+\|\omega\|_\infty)\varphi_1
\]
where $\varphi_1=|x|^{-\frac{2-b}{\alpha}}\chi_{\{|x|\leq A\}}\in L^s(\rd)$ with any $1\leq s<\frac{d\alpha}{2-b}=q_c$. Using \eqref{6}, for any choice of $s$ with $d/q_c<d/s<\tilde{s}_2+2$, we have  by Theorem \ref{ste} and \eqref{6} that
\[
t^{\frac{d}{2}(\frac{1}{s}-\frac{1}{r})}\|e^{-t\la}(\varphi-\psi)\|_r\lesssim\| \varphi-\psi\|_s\leq(c+\|\omega\|_\infty)\|\varphi_1\|_s<\infty.
\]
This implies (letting $\delta:=\frac{d}{2}(\frac{1}{s}-\frac{1}{r})-\beta(r)=\frac{d}{2s}-\frac{2-b}{2\alpha}$) %\textcolor{red}{following step seems not correct as we shall need $\frac{\tilde{s_2} +2}{d}< \frac{1}{s}$ while we have other way inequality, please double check}
\begin{equation}\label{12}
\sup_{t>0} t^{\beta(r)+\delta}\|e^{-t\la}(\varphi-\psi)\|_r<\infty\quad\text{for all }0<\delta<\frac{\tilde{s}_2+2}{2}-\frac{2-b}{2\alpha}.
\end{equation}
%(In fact $\delta=\frac{d}{2}(\frac{1}{s}-\frac{1}{r})-\beta(r)=\frac{d}{2s}-\frac{2-b}{2\alpha}<\frac{\tilde{s}_2+2}{2}-\frac{2-b}{2\alpha}$)
%But for such $\delta$, $\beta(r)(\alpha + 1) + \delta\leq\frac{2-b}{2}-\frac{d\alpha}{2r}+\frac{2-b}{2\alpha}-\frac{d}{2r}+\frac{d}{2}-\frac{2-b}{2\alpha}=\frac{2-b}{2}-\frac{d\alpha}{2r}-\frac{d}{2r}+\frac{d}{2}=\frac{2-b}{2}-\frac{d}{2}(\frac{\alpha+1}{r}-1)$
Note that $\beta(r)(\alpha + 1)<1$ and therefore for $0<\delta\leq\frac{\tilde{s}_2+2}{2}-\frac{2-b}{2\alpha}$ small enough one has $\beta(r)(\alpha + 1)+\delta<1$, applying Theorem \ref{global2} (specifically \eqref{9}) we have
\begin{equation}\label{10}
\sup_{t>0} t^{\beta(r)+\delta}\|u(t)-u_\mathcal{S}(t)\|_r<\infty
\end{equation}
for $\delta>0$ small enough. This proves the result for $q=r$.

Note that
\[
u(t) - u_\mathcal{S}(t) = e^{-\frac{t}{2}\la }(u(t/2) - u_\mathcal{S}(t/2)) + \mu\int_{t/2}^te^{-(t-\sigma )\la}(|\cdot|^{-b} (|u(\sigma)|^\alpha u(\sigma) -|u_\mathcal{S}(\sigma)|^\alpha u_\mathcal{S}(\sigma)))d\sigma
\]
Therefore for $q\in[r,\frac{d}{\tilde{s}_1})$   using \eqref{6}, one has \[\tilde{s}_1<\frac{d}{q}<\frac{d}{r} <\tilde{s}_2+2\qquad\text{and}\qquad \tilde{s}_1<\frac{d}{q}<b+\frac{d(\alpha+1)}{q}<\tilde{s}_2+2.\]
Thus applying Theorem A with the pair $(r,q)$ and Proposition \ref{est0} with the pair $(\frac{r}{\alpha+1},q)$ we get
\begin{eqnarray*}
t^{\beta(q)+\delta}\|u(t) - u_\mathcal{S}(t)\|_{q}&\lesssim& t^{\beta(q)+\delta-\frac{d}{2}(\frac{1}{r}-\frac{1}{q})}\|u(t/2) - u_\mathcal{S}(t/2)\|_r+t^{\beta(q)+\delta}\\
&&\cdot\int_{t/2}^t(t-\sigma)^{-\frac{d}{2}(\frac{\alpha+1}{q}-\frac{1}{q})-\frac{b}{2}}\||u(\sigma)|^\alpha u(\sigma) -|u_\mathcal{S}(\sigma)|^\alpha u_\mathcal{S}(\sigma)\|_{\frac{q}{\alpha+1}}d\sigma.
\end{eqnarray*}
Using \eqref{10} for the first term in the LHS we have
\begin{eqnarray*}
t^{\beta(q)+\delta}\|u(t) - u_\mathcal{S}(t)\|_{q}&\lesssim&t^{\beta(q)+\delta-\frac{d}{2}(\frac{1}{r}-\frac{1}{q})-\beta(r)-\delta}+t^{\beta(q)+\delta}\\
&&\cdot\int_{t/2}^t(t-\sigma)^{-\frac{d\alpha}{2q}-\frac{b}{2}}(\|u(\sigma)\|_{q}^\alpha+\| u_\mathcal{S}(\sigma)\|_{q}^\alpha)\|u(\sigma) - u_\mathcal{S}(\sigma)\|_{q}d\sigma.
\end{eqnarray*}
Using Theorem \ref{global2} part \eqref{g24} we further have
\begin{eqnarray*}
t^{\beta(q)+\delta}\|u(t) - u_\mathcal{S}(t)\|_{q}
&\lesssim&t^{0}+C_M^\alpha t^{\beta(q)+\delta}\int_{t/2}^t(t-\sigma)^{-\frac{d\alpha}{2q}-\frac{b}{2}}\sigma^{-\frac{2-b}{2}+\frac{d\alpha}{2q}}\|u(\sigma) - u_\mathcal{S}(\sigma)\|_{q}d\sigma\\
&\lesssim&1+C_M^\alpha t^{\beta(q)+\delta}\int_{t/2}^t(t-\sigma)^{-\frac{d\alpha}{2q}-\frac{b}{2}}\sigma^{-\frac{2-b}{2}+\frac{d\alpha}{2q}-\beta(q)-\delta}\\
&&\cdot\sigma^{\beta(q)+\delta}\|u(\sigma) - u_\mathcal{S}(\sigma)\|_{q}d\sigma\\
&\lesssim&1+C_M^\alpha t^{\beta(q)+\delta}\int_{t/2}^t(t-\sigma)^{-\frac{d\alpha}{2q}-\frac{b}{2}}\sigma^{-\frac{2-b}{2}+\frac{d\alpha}{2q}-\beta(q)-\delta}d\sigma\\
&&\cdot\sup_{\tau>0}\tau^{\beta(q)+\delta}\|u(\tau) - u_\mathcal{S}(\tau)\|_{q}\\
&\lesssim&1+C_M^\alpha t^{0}\int_{1/2}^1(1-\sigma)^{-\frac{d\alpha}{2q}-\frac{b}{2}}\sigma^{-\frac{2-b}{2}+\frac{d\alpha}{2q}-\beta(q)-\delta}d\sigma\\
&&\cdot\sup_{\tau>0}\tau^{\beta(q)+\delta}\|u(\tau) - u_\mathcal{S}(\tau)\|_{q}
\end{eqnarray*}
where $q$ is so that $ $. Since $\frac{d\alpha}{2q}+\frac{b}{2}<\frac{d\alpha}{2r}+\frac{b}{2}<1$ we obtain %and $-\frac{2-b}{2}+\frac{d\alpha}{2q}-\beta(q)-\delta$  
\begin{equation*}
\sup_{t>0}t^{\beta(q)+\delta}\|u(t) - u_\mathcal{S}(t)\|_{q}\leq c+cC_M^\alpha  \sup_{t>0}t^{\beta(q)+\delta}\|u(t) - u_\mathcal{S}(t)\|_{q}
\end{equation*} and $C_M\to0$ as $M\to0$. Choosing $M>0$ small enough so that $cC_M^\alpha<1$, we achieve \begin{equation}\label{p7}
\sup_{t>0}t^{\beta(q)+\delta}\|u(t) - u_\mathcal{S}(t)\|_{q}\lesssim1 \Longrightarrow\|u(t) - u_\mathcal{S}(t)\|_{q}\leq ct^{-\beta(q)-\delta}\ \forall\ t>0.
\end{equation}

Now $u_{\mathcal{S}}(t,x)=\lambda^{\frac{2-b}{\alpha}}u_{\mathcal{S}}(\lambda^2t,\lambda x)$ for all $\lambda$, and hence taking $\lambda=\frac{1}{\sqrt{t}}$, we have $u_{\mathcal{S}}(t,x)=t^{-\frac{2-b}{2\alpha}}u_{\mathcal{S}}(1,\frac{x}{\sqrt{t}})$ and hence
\begin{align*}
\|u_{\mathcal{S}}(t)\|_q=t^{-\frac{2-b}{2\alpha}}\left(\int_{\rd}|u_{\mathcal{S}}(1,{x}/{\sqrt{t}})|^qdx\right)^{1/q}=t^{-\frac{2-b}{2\alpha}}\left(t^{\frac{d}{2}}\int_{\rd}|u_{\mathcal{S}}(1,y)|^qdy\right)^{1/q}=t^{-\beta(q)}\|u_{\mathcal{S}}(1)\|_q.
\end{align*} Therefore using \eqref{p7}, for large $t>0$ we have $\|u(t)-u_{\mathcal{S}}(t)\|_q\leq\frac{1}{2}\|u_{\mathcal{S}}(t)\|_q$ and thus
\begin{align}\label{p8}
\|u(t)\|_q\geq\|u_{\mathcal{S}}(t)\|_q-\|u(t) - u_\mathcal{S}(t)\|_q\geq \frac{1}{2}\|u_{\mathcal{S}}(t)\|_q=\frac{1}{2}t^{-\beta(q)}\|u_{\mathcal{S}}(1)\|_q.
\end{align}
Also
\begin{align}\label{p9}
\|u(t)\|_q\leq\|u_{\mathcal{S}}(t)\|_q+\|u(t) - u_\mathcal{S}(t)\|_q\leq \frac{3}{2}\|u_{\mathcal{S}}(t)\|_q=\frac{3}{2}t^{-\beta(q)}\|u_{\mathcal{S}}(1)\|_q.
\end{align}
This completes the proof.
\end{proof}

\subsection{Case II: Linear Behaviour}
For this case, we need the following technical result to be proved in the Appendix.
\begin{lemma}\label{11}
Assume that $0<b<\min(2,d)$ and $\frac{2-b}{\tilde{s}_2+2}<\alpha<\frac{2-b}{\tilde{s}_1}$. Let $\alpha_1$ be real number such that \[\max\left(\frac{2-b}{\tilde{s}_2+2},\frac{\tilde{s}_1\alpha}{\tilde{s}_2+2-b-\tilde{s}_1\alpha}\right)<\alpha_1 < \alpha<\frac{2-b}{\tilde{s}_1}.
\]
 Let $r_1$ be a real number satisfying
 \[
 \max\left(\frac{(\alpha_1+1)d}{\tilde{s}_2+2-b},\frac{d\alpha_1}{2-b}\right)<r_1<\min\left(\frac{d\alpha_1(\alpha_1+1)}{(2-b(\alpha_1+1))_+},\frac{d\alpha_1}{\tilde{s}_1\alpha}\right).
 \]
 Let \begin{align*}
  r_2&=\frac{\alpha}{\alpha_1}r_1\\
\beta_1&=\frac{2-b}{2\alpha_1}-\frac{d}{2r_1}\\
\beta_2&=\frac{2-b}{2\alpha}-\frac{d}{2r_2}\\
r_{12}&=\frac{\alpha+1}{\alpha_1+1}r_1\\
\beta_{12}&=\frac{\alpha_1+1}{\alpha+1}\beta_1.
 \end{align*}
Then one has the following
\begin{enumerate}
\item $\beta_1,\beta_2,\beta_{12}>0$
\item $\tilde{s}_1<\frac{d}{r_1}<b+\frac{(\alpha+1)d}{r_{12}}<\tilde{s}_2+2$,  $\tilde{s}_1<\frac{d}{r_2}<b+\frac{(\alpha+1)d}{r_2}<\tilde{s}_2+2$\label{11ii}
\item $ \frac{d}{2} (\frac{\alpha+1}{r_{12}} - \frac{1}{r_1} ) + \frac{b}{d} = \frac{d\alpha}{2r_2}  + \frac{b}{2} < 1$\label{11iii}
\item $\beta_2(\alpha+1),\beta_{12}(\alpha+1)<1$\label{11iv}
\item $\beta_2 -\frac{d\alpha}{2r_2}  -\frac{b}{2} - \beta_2(\alpha + 1) + 1 = 0$
\item $\beta_1- \frac{d}{2}  (\frac{\alpha+1}{r_{12}} - \frac{1}{r_1} ) -\frac{b}{2}  - \beta_{12}(\alpha + 1) + 1 = 0$.
\end{enumerate}
\end{lemma}
With the above, lemma in hand we now prove a variant of Theorem \ref{global2} which will be useful to prove Theorem \ref{asym} \eqref{asym2}.
\begin{theorem}\label{asym 1}
Let $0 < b < \min(2, d )$ and $\frac{2-b}{\tilde{s}_2+2}<\alpha<\frac{2-b}{\tilde{s}_1}$. Suppose that \[\max\left(\frac{2-b}{\tilde{s}_2+2},\frac{\tilde{s}_1\alpha}{\tilde{s}_2+2-b-\tilde{s}_1\alpha}\right)<\alpha_1 < \alpha<\frac{2-b}{\tilde{s}_1}.
 \]
Let $r_1, r_2, r_{12},\beta_1,\beta_2$ are real numbers as in Lemma \ref{11}. Suppose further that $M > 0$ satisfies the
inequality
$KM^\alpha <1$, where $K$ is a positive constant. Choose $R > 0$ such that
\[cR+KM^{\alpha+1} \leq M.
\]

Let $\varphi$ be a tempered distribution such that
\begin{equation}\label{16}
\sup_{t>0} t^{\beta_1} \|e^{-t\la}\varphi\|_{r_1} \leq R,\qquad \sup_{t>0} t^{\beta_2} \|e^{-t\la}\varphi\|_{r_2} \leq R.
\end{equation}
Then there exists a unique global solution $u$ of \eqref{0} such that
\begin{equation}\label{20}
\sup_{t>0} t^{\beta_1} \| u(t)\|_{r_1}\leq M ,\qquad \sup_{t>0} t^{\beta_2} \| u(t)\|_{r_2}\leq M.
 \end{equation}
 
Furthermore, \begin{enumerate}
%\item \label{as3}
%\item \label{as4}
%\item $\lim_{t\to0} u(t) = \varphi$ in $L^s(\rd)$ if $\varphi\in L^s(\rd)$ for $s$ satisfying  $\frac{2-b}{\alpha_1}\leq\frac{d}{s}<b+\frac{d(\alpha+1)}{r_{12}}<\tilde{s}_2+2$\label{as0}.
\item  $\sup_{t\geq t_q} t^{ \frac{2-b}{2\alpha_1}-\frac{d}{2q}} \|u(t)\|_q \leq C_M<\infty$, for all $q \in [r_1, \frac{d}{\tilde{s}_1})$ with $C_M\to0$ as $M\to0$ \label{as1}
\item  $\sup_{t>0} t^{ \frac{2-b}{2\alpha}-\frac{d}{2q}} \|u(t)\|_q\leq C_M < \infty$, for all $q \in [r_2,\frac{d}{\tilde{s}_1})$ with $C_M\to0$ as $M\to0$\label{as2}.
\end{enumerate}

%Moreover, let $\varphi,\psi$
\end{theorem}
\begin{proof}[{\bf Proof}]
Let 
$$B_M:=\{u:(0,\infty)\to L^r(\rd)): \sup_{t>0} t^{\beta_1}\|u(t)\|_{r_1}\leq M, \sup_{t>0} t^{\beta_2}\|u(t)\|_{r_2}\leq M\}$$ and
\[
d(u,v)=: \max\left(\sup_{t>0} t^{\beta_1}\|u(t)-v(t)\|_{r_1},\sup_{t>0} t^{\beta_2}\|u(t)-v(t)\|_{r_2}\right)
\]
and
\[
\mathcal{J}_\varphi(u)(t)=e^{-t\la}\varphi+\mu\int_0^t e^{-(t-s)\la}(|\cdot|^{-b}|u(s)|^\alpha u(s))ds.
\]for $\varphi$ satisfies \eqref{16} and $u\in B_M$. 

Let $\varphi,\psi$ satisfy \eqref{16} and $u,v\in B_M$, then using Lemma \ref{11}, Proposition \ref{est0} for $(p,q)=(\frac{r_{12}}{\alpha+1},r_1)$ we have
\begin{eqnarray*}
&&\|\mathcal{J}_\varphi(u)(t)-\mathcal{J}_\psi(v)(t)\|_{r_1}\\
&\lesssim&\|e^{-t\la}\varphi-e^{-t\la}\psi\|_{r_1}+\int_0^t \|e^{-(t-s)\la}(|\cdot|^{-b}[|u(s)|^\alpha u(s)-|v(s)|^\alpha v(s)])\|_{r_1}ds\\
&\lesssim&\|e^{-t\la}(\varphi-\psi)\|_{r_1}+\int_0^t (t-s)^{-\frac{d}{2}(\frac{\alpha+1}{r_{12}}-\frac{1}{r_1})-\frac{b}{2}}\||u(s)|^\alpha u(s)-|v(s)|^\alpha v(s)\|_{\frac{r_{12}}{\alpha+1}}ds\\
&\lesssim&\|e^{-t\la}(\varphi-\psi)\|_{r_1}+\int_0^t (t-s)^{-\frac{d}{2}(\frac{\alpha+1}{r_{12}}-\frac{1}{r_1})-\frac{b}{2}}(\|u(s)\|_{r_{12}}^\alpha+\|v(s)\|_{r_{12}}^\alpha)\|u(s)- v(s)\|_{r_{12}}ds\\
&\lesssim&\|e^{-t\la}(\varphi-\psi)\|_{r_1}+M^\alpha d(u,v)\int_0^t (t-s)^{-\frac{d}{2}(\frac{\alpha+1}{r_{12}}-\frac{1}{r_1})-\frac{b}{2}}s^{-(\alpha+1)\beta_{12}}ds
\end{eqnarray*} as $\frac{1}{r_{12}}=\frac{1/(\alpha+1)}{r_1}+\frac{\alpha/(\alpha+1)}{r_2}$ and hence for $u\in B_M$,
\[
\|u(s)\|_{r_{12}}\leq\|u(s)\|_{r_1}^{\frac{1}{\alpha+1}}\|u(s)\|_{r_2}^{\frac{\alpha}{\alpha+1}}\leq Ms^{-\beta_1/(\alpha+1)-\alpha\beta_2/(\alpha+1)}=Ms^{-\frac{\beta_1+\alpha\beta_2}{\alpha+1}}
=Ms^{-\beta_{12}},
\]
\[
\|u(s)-v(s)\|_{r_{12}}\leq\|u(s)-v(s)\|_{r_1}^{\frac{1}{\alpha+1}}\|u(s)-v(s)\|_{r_2}^{\frac{\alpha}{\alpha+1}}\leq d(u,v)s^{-\frac{\beta_1+\alpha\beta_2}{\alpha+1}}=d(u,v)s^{-\beta_{12}}.
\]
Thus
\begin{eqnarray}\label{17}
t^{\beta_1}\|\mathcal{J}_\varphi(u)(t)-\mathcal{J}_\psi(v)(t)\|_{r_1}&\lesssim&t^{\beta_1}\|e^{-t\la}(\varphi-\psi)\|_{r_1}+M^\alpha d(u,v) t^{\beta_1-\frac{d}{2}(\frac{\alpha+1}{r_{12}}-\frac{1}{r_1})-\frac{b}{2}-(\alpha+1)\beta_{12}+1}\nonumber\\
&&\cdot\int_0^1 (1-s)^{-\frac{d}{2}(\frac{\alpha+1}{r_{12}}-\frac{1}{r_1})-\frac{b}{2}}s^{-(\alpha+1)\beta_{12}}ds\nonumber\\
&\lesssim&t^{\beta_1}\|e^{-t\la}(\varphi-\psi)\|_{r_1}+M^\alpha d(u,v)
\end{eqnarray}using Lemma \ref{11} (3), (4), (6).
Now  using Lemma \ref{11} (2), Proposition \ref{est0} for $(p,q)=(\frac{r_2}{\alpha+1},r_2)$ we have
\begin{eqnarray*}
&&\|\mathcal{J}_\varphi(u)(t)-\mathcal{J}_\psi(v)(t)\|_{r_2}\\
&\lesssim&\|e^{-t\la}(\varphi-\psi)\|_{r_2}+\int_0^t \|e^{-(t-s)\la}(|\cdot|^{-b}[|u(s)|^\alpha u(s)-|v(s)|^\alpha v(s)])\|_{r_2}ds\\
&\lesssim&\|e^{-t\la}(\varphi-\psi)\|_{r_2}+\int_0^t(t-s)^{-\frac{d}{2}(\frac{\alpha+1}{r_2}-\frac{1}{r_2})-\frac{b}{2}} \||u(s)|^\alpha u(s)-|v(s)|^\alpha v(s)\|_{\frac{r_2}{\alpha+1}}ds\\
&\lesssim&\|e^{-t\la}(\varphi-\psi)\|_{r_2}+\int_0^t(t-s)^{-\frac{d\alpha}{2r_2}-\frac{b}{2}}(\|u(s)\|_{r_2}^\alpha+\|v(s)\|_{r_2}^\alpha)\|u(s)- v(s)\|_{r_2} ds\\
&\lesssim&\|e^{-t\la}(\varphi-\psi)\|_{r_2}+M^\alpha d(u,v)\int_0^t(t-s)^{-\frac{d\alpha}{2r_2}-\frac{b}{2}}s^{-(\alpha+1)\beta_2} ds
\end{eqnarray*}
and hence using Lemma \ref{11} (3), (4), (5) we get
\begin{eqnarray}\label{18}
&&t^{\beta_2}\|\mathcal{J}_\varphi(u)(t)-\mathcal{J}_\psi(v)(t)\|_{r_2}\nonumber\\
&\lesssim&t^{\beta_2}\|e^{-t\la}(\varphi-\psi)\|_{r_2}+M^\alpha d(u,v)t^{\beta_2-\frac{d\alpha}{2r_2}-\frac{b}{2}-(\alpha+1)\beta_2+1}\int_0^1(1-s)^{-\frac{d\alpha}{2r_2}-\frac{b}{2}}s^{-(\alpha+1)\beta_2} ds\nonumber\\
&\lesssim&t^{\beta_2}\|e^{-t\la}(\varphi-\psi)\|_{r_2}+M^\alpha d(u,v).
\end{eqnarray}
Using \eqref{17}, \eqref{18} we have 
\begin{align}\label{19}
&d(\mathcal{J}_\varphi(u),\mathcal{J}_\psi(v))\nonumber\\
&\leq c\max\left(\sup_{t>0}t^{\beta_1}\|e^{-t\la}(\varphi-\psi)\|_{r_1},\sup_{t>0}t^{\beta_2}\|e^{-t\la}(\varphi-\psi)\|_{r_2}\right)+KM^\alpha d(u,v).
\end{align}
Putting $\psi=0,v=0$ in \eqref{19} one has
\[
d(\mathcal{J}_\varphi(u),0)\leq c\max\left(\sup_{t>0}t^{\beta_1}\|e^{-t\la}\varphi\|_{r_1},\sup_{t>0}t^{\beta_2}\|e^{-t\la}\varphi\|_{r_2}\right)+KM^\alpha d(u,0)\leq cR+KM^{\alpha+1}\leq M
\] and hence $\mathcal{J}_\varphi(u)\in B_M$. Thus $\mathcal{J}_\varphi$ maps from $B_M$ to itself. Putting $\psi=\varphi$ in \eqref{19}
\[
d(\mathcal{J}_\varphi(u),\mathcal{J}_\varphi(v))\leq KM^\alpha d(u,v)<d(u,v)
\]and hence $\mathcal{J}_\varphi$ is a contraction in $B_M$. Thus \eqref{0} has a unique solution in $B_M$ satisfying \eqref{20}.

\eqref{as1} follows from Lemma \ref{c1} and iteration as in proof of Theorem \ref{global2}.
\eqref{as2} follows from Lemma \ref{23} and iteration as in proof of Theorem \ref{global2}.

\end{proof}

This is a variant of Lemma \ref{23} used in the above result.

\begin{lemma}[A priori estimate]\label{c1}
Suppose $s < q  $ and
\[
\tilde{s}_1 <\frac{d}{q}  <b+ \frac{d(\alpha+1)}{s} < \tilde{s}_2+2,\qquad\frac{d}{2}\big(\frac{\alpha+1}{s}-\frac{1}{q}\big)<1-\frac{b}{2}
\]
Assume $t_0\geq1$ and $u$ be solution to \eqref{0} %with data $\varphi\in \mathcal{S}'(\rd)$ 
satisfying 
\[
\sup_{t>t_0}t^{\frac{2-b}{2\alpha_1}-\frac{d}{2s}}\|u(t)\|_s\leq A<\infty
\]
then
\[
\sup_{t\geq 2t_0}t^{\frac{2-b}{2\alpha_1}-\frac{d}{2q}}\|u(t)\|_q\lesssim A(1+A^\alpha)=:C_A<\infty
\]
with $C_A\to0$ as $A\to0$.
\end{lemma}
\begin{proof}[{\bf Proof}]
Note that
\[
u(t)=e^{-\frac{t}{2}\la}u(t/2)+\mu\int_{t/2}^te^{-(t-\sigma )\la}(|\cdot|^{-b} (|u(\sigma)|^\alpha u(\sigma))d\sigma
\] 
and therefore using Theorem \ref{ste} for $(p,q)=(s,q)$ and Proposition \ref{est0} for $(p,q)=(\frac{s}{\alpha+1},q)$ we have 
\begin{eqnarray*}
\|u(t)\|_q&\lesssim&\|e^{-\frac{t}{2}\la}u(t/2)\|_q+\int_{t/2}^t\|e^{-(t-\tau)\la}(|\cdot|^{-b} (|u(\tau)|^\alpha u(\tau))\|_q d\tau\\
&\lesssim&t^{-\frac{d}{2}(\frac{1}{s}-\frac{1}{q})}\|u(t/2)\|_s+\int_{t/2}^t(t-\tau)^{-\frac{d}{2}(\frac{\alpha+1}{s}-\frac{1}{q})-\frac{b}{2}}\||u(\tau)|^\alpha u(\tau)\|_{\frac{s}{\alpha+1}}d\tau\\
&\lesssim&t^{\frac{d}{2q}-\frac{2-b}{2\alpha_1}}A+\int_{t/2}^t(t-\tau)^{-\frac{d}{2}(\frac{\alpha+1}{s}-\frac{1}{q})-\frac{b}{2}}\| u(\tau)\|_s^{\alpha+1}d\tau.
\end{eqnarray*}
Now with $\beta_1(s)=\frac{2-b}{2\alpha_1}-\frac{d}{2s}$, and using $t/2\geq t_0$
\begin{eqnarray*}
\|u(t)\|_q
&\lesssim&t^{\frac{d}{2q}-\frac{2-b}{2\alpha_1}}A+A^{\alpha+1}\int_{t/2}^t(t-\tau)^{-\frac{d}{2}(\frac{\alpha+1}{s}-\frac{1}{q})-\frac{b}{2}}\tau^{-\beta_1(s)(\alpha+1)}d\tau\\
&\lesssim&t^{\frac{d}{2q}-\frac{2-b}{2\alpha_1}}A+A^{\alpha+1}t^{1-\frac{d}{2}(\frac{\alpha+1}{s}-\frac{1}{q})-\frac{b}{2}-\beta_1(s)(\alpha+1)}\int_{1/2}^1(1-\sigma)^{-\frac{d}{2}(\frac{\alpha+1}{s}-\frac{1}{q})-\frac{b}{2}}\sigma^{-\beta_1(s)(\alpha+1)}d\sigma\\
&\lesssim&t^{\frac{d}{2q}-\frac{2-b}{2\alpha_1}}A+A^{\alpha+1}t^{\frac{d}{2q}-\frac{2-b}{2\alpha_1}-\frac{2-b}{2}(\frac{\alpha}{\alpha_1}-1)}%=A(1+A^\alpha)t^{\frac{d}{2q}-\frac{2-b}{2\alpha_1}}
\end{eqnarray*} as $\int_{1/2}^1(1-\sigma)^{-\frac{d}{2}(\frac{\alpha+1}{s}-\frac{1}{q})-\frac{b}{2}}\sigma^{-\beta(s)(\alpha+1)}d\sigma<\infty$.
This completes the proof as  $t\geq1$ for $t/2\geq t_0$.% for $t\geq1$.
\end{proof}
Following is again a technical result, to be used to prove Theorem \ref{b2} that further proves the final result.
\begin{lemma}\label{13}
Let $0 < b < \min(2, d)$ and $\frac{2-b}{\tilde{s}_2+2}<\alpha<\frac{2-b}{\tilde{s}_1}$. Let the real numbers $\alpha_1$ and $\alpha$ be such that\[\max\left(\frac{2-b}{\tilde{s}_2+2},\frac{\tilde{s}_1\alpha}{\tilde{s}_2+2-b-\tilde{s}_1\alpha}\right)<\alpha_1 < \alpha<\frac{2-b}{\tilde{s}_1}.
\]
Let $r_1,r_2,\beta_1,\beta_2$ are real numbers as in Lemma \ref{11}. Then there exists a real number $\delta_0 > 0$ such that, for all $0 < \delta < \delta_0$, there exists a real number $
         0<\theta_\delta <1$,
with the properties that, the two real numbers $\tilde{r}$ and $\tilde{\beta}$ given by
\begin{equation}\label{21}
\frac{1}{\tilde{r}}=\frac{\theta_\delta}{r_1}+\frac{1-\theta_\delta}{r_2},\qquad\tilde{\beta}=\theta_\delta\beta_1+(1-\theta_\delta)\beta_2
\end{equation}
satisfy the following conditions\begin{itemize}
\item[(i)]$\tilde{s}_1<\frac{d}{r_1} < b + \frac{d(\alpha+1)}{\tilde{r}} < \tilde{s}_2+2$
\item[(ii)]$\beta_1 +\delta-\frac{d}{2}(\frac{\alpha+1}{\tilde{r}}-\frac{1}{r_1})-\frac{b}{2}-\tilde{\beta}(\alpha+1)+1=0$
\item[(iii)]$\frac{d}{2}(\frac{\alpha+1}{\tilde{r}}-\frac{1}{r_1})$ $+\frac{b}{2}<1,\tilde{\beta}(\alpha+1)<1$.
\end{itemize}
Moreover this $\theta_\delta$ is given by
\begin{equation}\label{22}
\theta_\delta=\frac{1}{\alpha+1}+\frac{2\alpha_1\alpha}{(2-b)(\alpha-\alpha_1)(\alpha+1)}\delta.
\end{equation}
\end{lemma}
\begin{proof}[{\bf Proof}]
{\bf Step I:}   If \eqref{21} is true then (ii) is equlalent to \eqref{22}:\\
First note that 
% \begin{align*}
%\frac{d\alpha}{2\tilde{r}}+\tilde{\beta}=\theta_\delta(\frac{d\alpha}{2r_1}+\beta_1)+(1-\theta_\delta)(\frac{d\alpha}{2r_2}+\beta_2)=\theta_\delta\frac{2-b}{2\alpha_1}+(1-\theta_\delta)\frac{2-b}{2\alpha}
%\end{align*}
$\frac{\alpha+1}{\tilde{r}}-\frac{1}{r_1}=(\alpha+1)(\frac{1}{r_1}-\frac{1}{r_2})\theta_\delta+\frac{\alpha+1}{r_2}-\frac{1}{r_1}$, $\tilde{\beta}=\theta_\delta(\beta_1-\beta_2)+\beta_2$. Thus
\begin{align*}
&\frac{d}{2}(\frac{\alpha+1}{\tilde{r}}-\frac{1}{r_1})+\tilde{\beta}(\alpha+1)\\
=&\frac{d(\alpha+1)}{2}(\frac{1}{r_1}-\frac{1}{r_2})\theta_\delta+\frac{d(\alpha+1)}{2r_2}-\frac{d}{2r_1}+\theta_\delta(\beta_1-\beta_2)(\alpha+1)+\beta_2(\alpha+1).
\end{align*}
Then
\begin{align*}
&\beta_1 -\frac{d}{2}(\frac{\alpha+1}{\tilde{r}}-\frac{1}{r_1})-\tilde{\beta}(\alpha+1)\\
&=\beta_1-\frac{d(\alpha+1)}{2}(\frac{1}{r_1}-\frac{1}{r_2})\theta_\delta-\frac{d(\alpha+1)}{2r_2}+\frac{d}{2r_1}-\theta_\delta(\beta_1-\beta_2)(\alpha+1)-\beta_2(\alpha+1)\\
&=(\beta_1+\frac{d}{2r_1})-(\alpha+1)(\beta_2+\frac{d}{2r_2})-\theta_\delta(\alpha+1)(\frac{d}{2}(\frac{1}{r_1}-\frac{1}{r_2})+\beta_1-\beta_2)\\
&=\frac{2-b}{2\alpha_1}-(\alpha+1)\frac{2-b}{2\alpha}-\theta_\delta(\alpha+1)(\frac{2-b}{2\alpha_1}-\frac{2-b}{2\alpha})
\end{align*} 
Therefore (ii) is equivalent to
\begin{align*}
&\frac{2-b}{2\alpha_1}-(\alpha+1)\frac{2-b}{2\alpha}+\delta-\frac{b}{2}+1=\theta_\delta(\alpha+1)(\frac{2-b}{2\alpha_1}-\frac{2-b}{2\alpha})\\
&\Longleftrightarrow\frac{1}{\alpha_1}-\frac{\alpha+1}{\alpha}+\frac{2}{2-b}\delta+1=\theta_\delta(\alpha+1)(\frac{1}{\alpha_1}-\frac{1}{\alpha})\\
&\Longleftrightarrow\frac{1}{\alpha_1}-\frac{1}{\alpha}+\frac{2}{2-b}\delta=\theta_\delta(\alpha+1)(\frac{1}{\alpha_1}-\frac{1}{\alpha})
\end{align*}which is equivalent to \eqref{22}. 

{\bf Step II:} Validity of (i), (iii):

Note that \[\theta_\delta=\frac{1}{\alpha+1}+\epsilon(\delta)
\]
Now $\theta_0=\frac{1}{\alpha+1}+\epsilon(0)=\frac{1}{\alpha+1}$ and for this choice of $\theta_\delta$, one has $\tilde{r}=r_{12}$, $\tilde{\beta}=\beta_{12}$. So at $\theta=0$, the inequalities (i), (iii) hold by Lemma \ref{11}. Then by continuity of $\epsilon$ with respect to $\delta$ one has (i), (iii) for $\delta>0$ small enough. 
\end{proof}

\begin{theorem}\label{b2}
Let $0 < b < \min(2, d )$ and $\frac{2-b}{\tilde{s}_2+2}<\alpha<\frac{2-b}{\tilde{s}_1}$. Let the real numbers $\alpha_1$ and $\alpha$ be such that\[\max\left(\frac{2-b}{\tilde{s}_2+2},\frac{\tilde{s}_1\alpha}{\tilde{s}_2+2-b-\tilde{s}_1\alpha}\right)<\alpha_1 < \alpha<\frac{2-b}{\tilde{s}_1}.
\]
Let $r_1, r_2$ be two real numbers as in Lemma \ref{11}.   Let $\beta_1, \beta_2$ be given by Lemma \ref{11} and define $\beta_1(q)$ by
\[\beta_1(q)=\frac{2-b}{2\alpha_1}-\frac{d}{2q}, q>1
\]
Let $\psi(x) = \omega(x)|x|^{-\frac{2-b}{\alpha_1}}$ , where $\omega \in L^\infty(S^{d-1})$ is homogeneous of degree $0$. Let $\varphi\in C_0(\rd)$ be such that\[
|\varphi(x)|\leq c( 1+|x|^2)^{-\frac{2-b}{2\alpha_1}}\text{ for all }x\in\rd,\qquad |\varphi(x)|= \omega(x)|x|^{-\frac{2-b}{\alpha_1}}\text{ for all }|x|\geq A
\]for some constant $A > 0$, where c is a small positive constant and $\|\omega\|$ is sufficiently small.

Let u be the solution of \eqref{0} with initial data $\varphi$, constructed by Theorem \ref{asym 1} and let $w$ be the self-similar solution of \eqref{0} constructed by Theorem \ref{asym 1} with $\mu = 0$, and with initial data $\psi$. Then there exists $\delta_0 > 0$ such that for all $0<\delta<\delta_0$, and with $M$ perhaps smaller, there exists $C_\delta > 0$ such that
  \begin{equation}\label{28}
  \|u(t)- w(t)\|_q \leq C_\delta t^{-\beta_1(q)-\delta}, \forall t \geq t_q,
  \end{equation}
%  \[
% \|t^{\frac{2-b}{2\alpha_1}}u(t,\sqrt{t}\cdot)-w(1,\cdot)\|_q \leq C_\delta t^{-\delta} , \forall t>0,
% \]
for all $q \in [r_1, \frac{d}{\tilde{s}_1})$. In particular, if $\omega\neq  0$, %there exist d1 > 0, d2 > 0 two constants, such that d1t−β1(q) ≤ ∥u(t)∥q ≤ d2t−β1(q),
for large time $t$ %and for all r1 ≤ q ≤ ∞.
\[
c_1t^{-\beta_1(q)} \leq \|u(t)\|_q \leq c_2t^{-\beta_1(q)}.
\]
 
\end{theorem}
\begin{proof}[{\bf Proof}]
Let $\psi(x)=\omega(x)|x|^{-\frac{2-b}{2\alpha_1}}$, $x\in\rd$.
Let $\varphi_1=\varphi\chi_{\{|x|\leq1\}}$ and $\varphi_2=\varphi\chi_{\{|x|>1\}}$ so that $\varphi=\varphi_1+\varphi_2$.
Note that 
\[
|\varphi(x)|\leq c( 1+|x|^2)^{-\frac{2-b}{2\alpha_1}}\leq c|x|^{-\frac{2-b}{\alpha_1}}, \forall\ x\in\rd
\]
and thus $\varphi_1\in L^s(\rd)$ for $1\leq s<\frac{d\alpha_1}{2-b}$ and $\varphi_2\in L^\sigma(\rd)$ for $\frac{d\alpha_1}{2-b}<\sigma\leq\infty$.
Since $\alpha>\alpha_1$,
\[
|\varphi(x)|\leq c( 1+|x|^2)^{-\frac{2-b}{2\alpha_1}}\leq c( 1+|x|^2)^{-\frac{2-b}{2\alpha}}\leq c|x|^{-\frac{2-b}{\alpha}}, \forall\ x\in\rd.
\]
Then applying Theorem A
\begin{align*}
\|e^{-\la}\varphi\|_{r_1}&\leq \|e^{-\la}\varphi_1\|_{r_1}+\|e^{-\la}\varphi_2\|_{r_1}
\lesssim \|\varphi_1\|_s+\|\varphi_2\|_\sigma<\infty
\end{align*}by choosing $s,\sigma$ so that $\tilde{s}_1<\frac{d}{r_1}<\frac{d}{\sigma}<\frac{2-b}{\alpha_1}<\frac{d}{s}<\tilde{s}_2+2$. Similarly using $r_2>\frac{d\alpha}{2-b}$ one finds $\|e^{-\la}\varphi\|_{r_2}<\infty$. Then using the homogeneity of $|\cdot|^{-\frac{2-b}{\alpha_1}}$, $|\cdot|^{-\frac{2-b}{\alpha}}$ (and positivity of $e^{-t\la}$) we achieve
\[
\sup_{t>0}t^{\beta_1}\|e^{-t\la}\varphi\|_{r_1}\leq R,\qquad \sup_{t>0}t^{\beta_2}\|e^{-t\la}\varphi\|_{r_2}\leq R
\]after possibly choosing $c$ smaller. 

Proceeding as \eqref{12}, one finds \begin{equation}\label{14}
\sup_{t>0} t^{\beta_1+\delta}\|e^{-t\la}(\varphi-\psi)\|_{r_1}<\infty\quad\text{for all }0<\delta\leq\frac{\tilde{s}_2+2}{2}-\frac{2-b}{2\alpha_1}.
\end{equation}
for $\delta>0$ small enough.

Let $v(t)=e^{-t\la}\varphi$ then $u(t)=v(t)+\mu\int_0^te^{-(t-\sigma )\la}(|\cdot|^{-b} (|u(\sigma)|^\alpha u(\sigma))d\sigma$, therefore using Proposition \ref{est0}, Lemma \ref{13}
\begin{eqnarray*}
\|u(t)-v(t)\|_{r_1}&\lesssim&\int_0^t\|e^{-(t-\sigma )\la}(|\cdot|^{-b} (|u(\sigma)|^\alpha u(\sigma))\|_{r_1}d\sigma\\
&\lesssim&\int_0^t(t-\sigma )^{-\frac{d}{2}(\frac{\alpha+1}{\tilde{r}}-\frac{1}{r_1})-\frac{b}{2}}\||u(\sigma)|^\alpha u(\sigma))\|_{\frac{\tilde{r}}{\alpha+1}}d\sigma\\
&=&\int_0^t(t-\sigma )^{-\frac{d}{2}(\frac{\alpha+1}{\tilde{r}}-\frac{1}{r_1})-\frac{b}{2}}\| u(\sigma))\|_{\tilde{r}}^{\alpha+1}d\sigma.
\end{eqnarray*}
Now note that
\[
\| u(\sigma))\|_{\tilde{r}}\leq\| u(\sigma))\|_{r_1}^{\theta_\delta}\| u(\sigma))\|_{r_2}^{1-\theta_\delta}
\leq M \sigma^{-\beta_1\theta_\delta-\beta_2(1-\theta_\delta)}=M \sigma^{-\tilde{\beta}}\]
and hence
\begin{eqnarray*}
t^{\beta_1+\delta}\|u(t)-v(t)\|_{r_1}&\lesssim&M^{\alpha+1}t^{\beta_1+\delta}\int_0^t(t-\sigma )^{-\frac{d}{2}(\frac{\alpha+1}{\tilde{r}}-\frac{1}{r_1})-\frac{b}{2}}\sigma^{-(\alpha+1)\tilde{\beta}}d\sigma\\
&\lesssim&M^{\alpha+1}t^{\beta_1+\delta+1-\frac{d}{2}(\frac{\alpha+1}{\tilde{r}}-\frac{1}{r_1})-\frac{b}{2}-(\alpha+1)\tilde{\beta}}\int_0^1(1-\sigma )^{-\frac{d}{2}(\frac{\alpha+1}{\tilde{r}}-\frac{1}{r_1})-\frac{b}{2}}\sigma^{-(\alpha+1)\tilde{\beta}}d\sigma\\
&=&M^{\alpha+1}\int_0^1(1-\sigma )^{-\frac{d}{2}(\frac{\alpha+1}{\tilde{r}}-\frac{1}{r_1})-\frac{b}{2}}\sigma^{-(\alpha+1)\tilde{\beta}}d\sigma<\infty
\end{eqnarray*}using Lemma \ref{13}. This together with \eqref{14} imply
\begin{equation}\label{15}
\|u(t)-w(t)\|_{r_1}\leq\|u(t)-v(t)\|_{r_1}+\|e^{-t\la}\varphi-e^{-t\la}\psi\|_{r_1}\lesssim t^{-\beta_1-\delta}
\end{equation}
which proves \eqref{28} for $q=r_1$.

Now 
\[
u(t)-w(t)=e^{-\frac{t}{2}\la}(u(t/2)-w(t/2))+\mu\int_{t/2}^te^{-(t-\sigma )\la}(|\cdot|^{-b} (|u(\sigma)|^\alpha u(\sigma))d\sigma
\]
%and using Lemma \ref{13} for $q\in[r_1,\frac{d}{\tilde{s}_1})$,
%\[
%\tilde{s}_1<\frac{d}{q}<\frac{d}{r_1} <\tilde{s}_2+2\qquad\text{and}\qquad \tilde{s}_1<\frac{d}{q}<\frac{d}{r_1}<b+\frac{d(\alpha+1)}{\tilde{r}}<\tilde{s}_2+2.
%\]
% hence using Proposition \ref{est0}
%\begin{eqnarray*}
%t^{\beta_1(q)+\delta}\|u(t)-w(t)\|_q&\lesssim& t^{\beta_1(q)+\delta-\frac{d}{2}(\frac{1}{r_1}-\frac{1}{q})}\|u(t/2) - w(t/2)\|_{r_1}+t^{\beta_1(q)+\delta}\\
%&&\cdot\int_{t/2}^t(t-\sigma)^{-\frac{d}{2}(\frac{\alpha+1}{\tilde{r}}-\frac{1}{q})-\frac{b}{2}}\||u(\sigma)|^\alpha u(\sigma) \|_{\frac{\tilde{r}}{\alpha+1}}d\sigma,\\
%\end{eqnarray*}
%using \eqref{15} 
%\begin{eqnarray*}
%t^{\beta_1(q)+\delta}\|u(t)-w(t)\|_q
%&\lesssim& t^{\beta_1(q)-\beta_1-\frac{d}{2}(\frac{1}{r_1}-\frac{1}{q})}+t^{\beta_1(q)+\delta}\\
%&&\cdot\int_{t/2}^t(t-\sigma)^{-\frac{d}{2}(\frac{\alpha+1}{\tilde{r}}-\frac{1}{q})-\frac{b}{2}}\|u(\sigma) \|_{\tilde{r}}^{\alpha+1}d\sigma\\
%&\lesssim& t^{0}+M^{\alpha+1}t^{\beta_1(q)+\delta}\int_{t/2}^t(t-\sigma)^{-\frac{d}{2}(\frac{\alpha+1}{\tilde{r}}-\frac{1}{q})-\frac{b}{2}}\sigma^{-(\alpha+1)\tilde{\beta}} d\sigma\\
%&\lesssim& t^{0}+M^{\alpha+1}t^{\beta_1(q)+\delta+1-\frac{d}{2}(\frac{\alpha+1}{\tilde{r}}-\frac{1}{q})-\frac{b}{2}-(\alpha+1)\tilde{\beta}}\\
%&&\cdot\int_{1/2}^1(1-\sigma)^{-\frac{d}{2}(\frac{\alpha+1}{\tilde{r}}-\frac{1}{q})-\frac{b}{2}}\sigma^{-(\alpha+1)\tilde{\beta}} d\sigma
%\end{eqnarray*}
%======
using Lemma \ref{13} for $q\in[r_2,\frac{d}{\tilde{s}_1})$,
\[
\tilde{s}_1<\frac{d}{q}<\frac{d}{r_1} <\tilde{s}_2+2\qquad\text{and}\qquad \tilde{s}_1<\frac{d}{q}<b+\frac{d(\alpha+1)}{q}<b+\frac{d(\alpha+1)}{r_2}
<\tilde{s}_2+2.
\]
\begin{eqnarray*}
&&\|u(t)-w(t)\|_q\\
&\lesssim& t^{-\frac{d}{2}(\frac{1}{r_1}-\frac{1}{q})}\|u(t/2) - w(t/2)\|_{r_1}+\int_{t/2}^t(t-\sigma)^{-\frac{d}{2}(\frac{\alpha+1}{q}-\frac{1}{q})-\frac{b}{2}}\||u(\sigma)|^\alpha u(\sigma) \|_{\frac{q}{\alpha+1}}d\sigma\\
&\lesssim& t^{-\frac{d}{2}(\frac{1}{r_1}-\frac{1}{q})-\beta_1-\delta}+\int_{t/2}^t(t-\sigma)^{-\frac{d\alpha}{2q}-\frac{b}{2}}\|u(\sigma) \|_q^{\alpha+1} d\sigma.
\end{eqnarray*}using \eqref{15}.
Now  from \eqref{as1}, \eqref{as2} in Theorem \ref{asym 1}, for $\sigma\geq t_q$, $\|u(\sigma) \|_q\leq C_M\sigma^{-\beta_1(q)}$ as well as $\|u(\sigma) \|_q\leq C_M\sigma^{-\beta(q)}$ for all $q\in[r_2,\frac{d}{\tilde{s}_1})$. Therefore for $\sigma\geq t_q$, 
\[
\|u(\sigma) \|_q^{\alpha+1}\leq C_M^{\alpha+1}\sigma^{-(\alpha+1)[\beta_1(q)\theta_\delta+\beta(q)(1-\theta_\delta)]}
\]
 and $\beta_1(q)\theta_\delta+\beta(q)(1-\theta_\delta)=\beta_1\theta_\delta+\beta_2(1-\theta_\delta)+\frac{d}{2}(\frac{1}{r_1}-\frac{1}{q})\theta_\delta+\frac{d}{2}(\frac{1}{r_2}-\frac{1}{q})(1-\theta_\delta)=\tilde{\beta}+\frac{d}{2}(\frac{\theta_\delta}{r_1}+\frac{1-\theta_\delta}{r_1})-\frac{d}{2q}=\tilde{\beta}+\frac{d}{2\tilde{r}}-\frac{d}{2q}$.
  Therefore for $t\geq 2t_q$, 
\begin{eqnarray*}
&&t^{\beta_1(q)+\delta}\|u(t)-w(t)\|_q\\
&\lesssim& t^{\beta_1(q)-\beta_1-\frac{d}{2}(\frac{1}{r_1}-\frac{1}{q})}+t^{\beta_1(q)+\delta}\int_{t/2}^t(t-\sigma)^{-\frac{d\alpha}{2q}-\frac{b}{2}}\|u(\sigma) \|_q^{\alpha+1} d\sigma\\
&\lesssim& t^0+C_M^{\alpha+1}t^{\beta_1(q)+\delta}\int_{t/2}^t(t-\sigma)^{-\frac{d\alpha}{2q}-\frac{b}{2}}\sigma^{-(\alpha+1)(\tilde{\beta}+\frac{d}{2\tilde{r}}-\frac{d}{2q})} d\sigma\\
&\lesssim& 1+C_M^{\alpha+1}t^{\beta_1(q)+\delta-\frac{d\alpha}{2q}-\frac{b}{2}-(\alpha+1)(\tilde{\beta}+\frac{d}{2\tilde{r}}-\frac{d}{2q})+1}\int_{1/2}^1(1-\sigma)^{-\frac{d\alpha}{2q}-\frac{b}{2}}\sigma^{-(\alpha+1)(\tilde{\beta}+\frac{d}{2\tilde{r}}-\frac{d}{2q})} d\sigma.
\end{eqnarray*}
But $\beta_1(q)+\delta-\frac{d\alpha}{2q}-\frac{b}{2}-(\alpha+1)(\tilde{\beta}+\frac{d}{2\tilde{r}}-\frac{d}{2q})+1=\beta_1(q)+\delta-\frac{d}{2}(\frac{\alpha+1}{q}-\frac{1}{q})-\frac{b}{2}-(\alpha+1)(\tilde{\beta}+\frac{d}{2\tilde{r}}-\frac{d}{2q})+1=\beta_1+\frac{d}{2}(\frac{1}{r_1}-\frac{1}{q})+\delta-\frac{d}{2}(\frac{\alpha+1}{\tilde{r}}-\frac{1}{r_1})-\frac{d}{2}(\frac{1}{r_1}-\frac{1}{q})-\frac{b}{2}-(\alpha+1)(\tilde{\beta}+\frac{d}{2q}-\frac{d}{2q})+1=\beta_1+\delta-\frac{d}{2}(\frac{\alpha+1}{\tilde{r}}-\frac{1}{r_1})-\frac{b}{2}-(\alpha+1)\tilde{\beta}+1=0$ and $\frac{d\alpha}{2q}+\frac{b}{2}<1$. Thus $t^{\beta_1(q)+\delta}\|u(t)-w(t)\|_q\lesssim1$ for $t\geq 2t_q$.
This proves \eqref{28} for $q\in[r_2,\frac{d}{\tilde{s}_1})$. To prove  \eqref{28} for $(r_1,r_2)$, we use interpolation. The final conclusion follows as in the proof in nonlinear case, see \eqref{p8}, \eqref{p9} in Subsection \ref{AB1}.
\end{proof}
\begin{remark}\label{r2}
After proving \eqref{28} for $q=r_1$ (i.e. \eqref{15}) the authors in \cite{slimene2017well} proves it for $q=\infty$ and interpolates them to achieve \eqref{28} for $q\in(r_1,\infty)$. Since we do not have the decay estimate \eqref{d} for $q=\infty$, we could not achieve \eqref{28} for $q=\infty$ and hence we need to take a different path to achieve the result.% could not use interpolation to get the . 
\end{remark}

%\begin{proof}[{\bf Proof}]
%
%\end{proof}
\begin{proof}[{\bf Proof of Theorem \ref{asym} \eqref{asym2}}]
Since $\sigma>\frac{2-b}{\alpha}$, we can find $\alpha_1<\alpha$ so that $\sigma=\frac{2-b}{\alpha_1}$. To apply Theorem \ref{b2} we need
\[\max\left(\frac{2-b}{\tilde{s}_2+2},\frac{\tilde{s}_1\alpha}{\tilde{s}_2+2-b-\tilde{s}_1\alpha}\right)<\alpha_1 < \alpha<\frac{2-b}{\tilde{s}_1}.
\]which will follow if 
$\frac{2-b}{\tilde{s}_2+2}<\frac{2-b}{\sigma}$ i.e. $\sigma< \tilde{s}_2+2$ and $\frac{\tilde{s}_1\alpha}{\tilde{s}_2+2-b-\tilde{s}_1\alpha}<\frac{2-b}{\sigma}$ i.e. $\sigma<\frac{2-b}{\tilde{s}_1\alpha}(\tilde{s}_2+2-b-\tilde{s}_1\alpha)$.
\end{proof}
\section*{Appendix}
\renewcommand*{\theAL}{A\arabic{AL}}
\begin{AL}\label{a1}
Let $\max\left(\frac{2-b}{\tilde{s}_2+2},\frac{\tilde{s}_1\alpha}{\tilde{s}_2+2-b-\tilde{s}_1\alpha}\right)<\alpha_1<\alpha<\frac{2-b}{\tilde{s}_1}$. Then one can choose $r_1$ so that \begin{itemize}
\item[(i)] $\tilde{s}_1<\frac{d\alpha_1}{r_1\alpha}<\frac{d}{r_1}<b+\frac{(\alpha_1+1)d}{r_1}<\tilde{s}_2+2$
\item[(ii)]$\frac{d\alpha_1}{2r_1}  + \frac{b}{2} < 1$
\item[(iii)] $\beta_1(\alpha_1+1)<1$.
\end{itemize}
\end{AL}
\begin{proof}[{\bf Proof}]
(i), (ii), (iii) are equivalent to $\frac{(\alpha_1+1)d}{\tilde{s}_2+2-b}<r_1<\frac{d\alpha_1}{\tilde{s}_1\alpha}$, $\frac{d\alpha_1}{2-b}<r_1$, $r_1<\frac{d\alpha_1(\alpha_1+1)}{2-b(\alpha_1+1)}$ if $2-b(\alpha_1+1)>0$. Then to make room for $r_1$, one thus needs \begin{equation*}
\begin{rcases}
\frac{(\alpha_1+1)d}{\tilde{s}_2+2-b}\\
\frac{d\alpha_1}{2-b}
\end{rcases}
<r_1<
\begin{cases}\frac{d\alpha_1}{\tilde{s}_1\alpha}\\
\frac{d\alpha_1(\alpha_1+1)}{2-b(\alpha_1+1)}
\end{cases}
\end{equation*}which is possible if $\max(\frac{2-b}{\tilde{s}_2+2},\frac{\tilde{s}_1\alpha}{\tilde{s}_2+2-b-\tilde{s}_1\alpha})<\alpha_1<\alpha<\frac{2-b}{\tilde{s}_1}$.
\end{proof}

%\textcolor{red}{what is missing the lemma below?}

\begin{AL}\label{a2}
Let $\max\left(\frac{2-b}{\tilde{s}_2+2},\frac{\tilde{s}_1\alpha}{\tilde{s}_2+2-b-\tilde{s}_1\alpha}\right)<\alpha_1<\alpha<\frac{2-b}{\tilde{s}_1}$.
Let $r_1$, $r_2$ be as in Lemma \ref{11}. Then 
\begin{itemize}
\item[(i)] $r_1<r_2$
\item[(ii)] $\tilde{s}_1<\frac{d}{r_1}<b +\frac{(\alpha_1+1)d}{r_1} <\tilde{s}_2+2$, $\tilde{s}_1<\frac{d}{r_2}<b +\frac{(\alpha+1)d}{r_2} <\tilde{s}_2+2$
\item[(iii)] $ \frac{d\alpha_1}{2r_1}  + \frac{b}{2} < 1$, $\frac{d\alpha}{2r_2}  + \frac{b}{2} < 1$
\item[(iv)] $\beta_1(\alpha_1+1)<1$, $\beta_2(\alpha+1)<1$
%\item[(v)]
%\item[(vi)]
\end{itemize}
\end{AL}
\begin{proof}[{\bf Proof}]
(i) is follows from $\alpha>\alpha_1$. 

First part of (ii) is a consequence of Lemma \ref{a1}. Note that $\frac{(\alpha+1)\alpha_1}{\alpha(\alpha_1+1)}<1$ and hence $\frac{(\alpha+1)}{r_2}<\frac{(\alpha_1+1)}{r_1}$. This together with $\tilde{s}_1<\frac{d\alpha_1}{r_1\alpha}$ %and $\frac{(\alpha+1)}{r_2}<\frac{(\alpha_1+1)}{r_1}$
imply the second part of (ii). 

(iii) follows from (i) and Lemma \ref{a1} (ii).

First one in (iv) is exactly part (iii) in Lemma \ref{a1}. Since $\beta_2(\alpha+1)=\beta_1(\alpha_1+1)\frac{(\alpha+1)\alpha_1}{\alpha(\alpha_1+1)}$ the last inequality in (iv) follows from the fact $\frac{(\alpha+1)\alpha_1}{\alpha(\alpha_1+1)}<1$.
\end{proof}

\begin{proof}[{\bf Proof of Lemma \ref{11}}]
(1) $\beta_1>0$ follows from $r_1>\frac{d\alpha}{2-b}$. Then $\beta_2,\beta_{12}>0$ follows from their definitions.

(2), (3), (4) are essentially consequences parts (ii), (iii), (iv) of Lemma \ref{a2} respectively.
 
 (5), (6) follows by simple computation.
\end{proof}

%\textcolor{red}{make references consistent...}

\noindent
{\textbf{Acknowledgement}:} D.G. B is thankful to DST-INSPIRE (DST/INSPIRE/04/2016/001507) for the research grant. S. H. acknowledges Dept of Atomic Energy, Govt of India, for the financial support and Harish-Chandra Research Institute for the research facilities provided. 
\bibliographystyle{siam}
\bibliography{BH-heat}
\end{document}